\documentclass[12pt]{amsart}
\usepackage{amsmath}
\usepackage{amssymb}
\usepackage{epsfig}
\usepackage{wasysym}
\usepackage{graphicx}
\usepackage{bm}
\usepackage{xcolor}

\numberwithin{equation}{section}

\input xy
\xyoption{all}

\calclayout
\allowdisplaybreaks[3]

\theoremstyle{plain}
\newtheorem{prop}{Proposition}

\newtheorem{theo}[prop]{Theorem}
\newtheorem{coro}[prop]{Corollary}

\newtheorem{lemm}[prop]{Lemma}

\theoremstyle{definition}

\newtheorem{ques}[prop]{Question}
\newtheorem{conj}[prop]{Conjecture}

\newtheorem{rema}[prop]{Remark}

\newtheorem{exam}[prop]{Example}

\def\lra{\longrightarrow}
\def\ra{\rightarrow}

\def\bA{{\mathbb A}}

\def\bP{{\mathbb P}}
\def\bQ{{\mathbb Q}}

\def\bZ{{\mathbb Z}}
\def\bR{{\mathbb R}}

\def\bC{{\mathbb C}}

\def\bF{{\mathbb F}}

\def\codim{\mathrm{codim}}

\def\Bl{\mathrm{Bl}}

\def\Aut{\mathrm{Aut}}

\def\Hom{\mathrm{Hom}}

\def\Bir{\mathrm{Bir}}

\def\lim{\mathrm{lim}}

\def\Sym{\mathrm{Sym}}

\def\Spec{\mathrm{Spec}}
\def\ori{\mathrm{or}}

\makeatother
\makeatletter

\author{Maxim Kontsevich}
\address{Institut des Hautes \'Etudes Scientifiques, 
35 route de Chartres, 91440 Bures-sur-Yvette, France}
\email{maxim@ihes.fr}

\author{Vasily Pestun}
\address{Institut des Hautes \'Etudes Scientifiques, 
35 route de Chartres, 91440 Bures-sur-Yvette, France}
\email{vasily.pestun@ihes.fr }

\author{Yuri Tschinkel}
\address{Courant Institute\\
                New York University \\
                New York, NY 10012 \\
                USA }
\email{tschinkel@cims.nyu.edu}

\address{Simons Foundation\\
160 Fifth Avenue\\
New York, NY 10010\\
USA}

\title[Equivariant birational geometry]{Equivariant birational geometry and modular symbols}

\begin{document}
\date{August 2019}

\begin{abstract}
We introduce new invariants in equivariant birational geometry and study their relation to modular symbols and cohomology of arithmetic groups. 
\end{abstract}

\maketitle

\section{Introduction}
\label{sect:intro}

Let $G$ be a finite abelian group and 
$$
A=G^\vee = \Hom(G, \bC^\times)
$$
the group of characters of $G$. Fix an integer $n\ge 2$. 
Consider the $\bZ$-module
$$
\mathcal B_{n}(G)
$$
generated by symbols
$$
[a_1,\ldots, a_n], \quad a_i\in A,
$$
such that $a_1, \ldots, a_n$ generate $A$, i.e., 
$$
\sum_i \bZ a_i =A,
$$
and subject to relations:
\begin{itemize}
\item[(S)]
for all permutations $\sigma \in \mathfrak S_n$ and all $a_1,\ldots, a_n\in A$ we have
$$
[a_{\sigma(1)} , \ldots, a_{\sigma(n)} ] = [a_1,\ldots, a_n],
$$ 
\item[(B)] 
for all $2\le k \le n$, all $a_1,\ldots, a_k\in A$, and all 
$b_1,\ldots, b_{n-k} \in A$ such that 
$$
\sum_i \bZ a_i + \sum_j \bZ b_j = A
$$
we have
$$
[a_1,\ldots ,a_k, b_1, \ldots , b_{n-k}] =
$$
$$
= \sum_{1\le i\le k, \,\, a_i\neq a_{i'}, \forall i'<i}    
[a_1-a_i, \ldots , a_i (\text{on $i$-th place}),  \ldots, a_k-a_i, b_1, \ldots, b_{n-k}]
$$
\end{itemize}
We have,
$$
\mathcal B_1(G) = \begin{cases} \bZ^{\phi(N)} & \text{if } G=\bZ/N\bZ, N\ge 1 \\ 0 & \text{otherwise.}\end{cases}
$$
For example, for $n=4$ and $k=3$ and $a_1=a_2=a$ and $a_3=a'\neq a$ 
and $b_1=b$, the relation translates to 
\begin{equation}\label{eqn:n4k3B}
[a,a,a',b] = [a,0,a'-a,b] + [a-a', a-a', a',b].
\end{equation}
When $n=2$ there is only one possibility for $k$, namely, $k=2$.

\begin{exam}
\label{exam:1} 
The group 
$\mathcal B_2(G)$ is generated by symbols $[a_1,a_2]$ such that 
$$
a_1,a_2\in \bZ/N\bZ, \quad   \gcd(a_1,a_2,N)=1,
$$ 
and subject to relations
\begin{itemize}
\item $[a_1,a_2]= [a_2,a_1]$, 
\item $[a_1,a_2]= [a_1,a_2-a_1] + [a_1-a_2,a_2]$, where $a_1\neq a_2$,
\item $[a,a]= [a,0]$, for all $a\in \bZ/N\bZ$, $\gcd(a,N)=1$.
\end{itemize} 
For $p\ge 5$ a prime, the $\bQ$-rank of $\mathcal B_2(\bZ/p\bZ)$ equals
\begin{equation}
\label{eqn:23}
\frac{p^2+23}{24}.
\end{equation}
For us, this was the first sign that automorphic forms play a role in this theory. 
We will discuss the connection to modular symbols in Section~\ref{sect:modular}.
\end{exam}

\begin{rema}
\label{rema:comment}
The group $\mathcal B_2(\bZ/p\bZ)$ can have torsion, e.g., 
for $p=37$, there is $\ell$-torsion for $\ell=3$ and $19$.  

For $n\ge 3$, the system of relations in $\mathcal B_n(G)$ is highly overdetermined. Nevertheless, 
computer experiments show that nontrivial solutions exist, e.g., 
for  $G=\bZ/27\bZ$ or $\bZ/43 \bZ$, the $\bQ$-rank of $\mathcal B_4(G)$ equals 1. 
\end{rema}

Let $X$ be a smooth irreducible projective algebraic variety of dimension $n\ge 2$, 
over a fixed algebraically closed field of characteristic zero (e.g., $\bC$), 
equipped with a birational, generically free action of $G$. 
After $G$-equivariant resolution of singularities, we may assume that the 
action of $G$ is regular. To such an $X$ we associate an element of $\mathcal B_n(G)$ as follows: 
Let 
\begin{equation}
\label{eqn:dec}
X^G=\coprod_{\alpha\in \mathcal{A}} \,  F_{\alpha}
\end{equation}
be the $G$-fixed point locus; it is a disjoint union of closed smooth irreducible subvarieties of $X$. 
Put
$$
\dim(F_{\alpha})=n_{\alpha}\le n-1.
$$
On each irreducible component $F_{\alpha}$ we fix a point $x_{\alpha} \in F_{\alpha}$  
and consider the action of $G$ in its tangent space 
$\mathcal T_{x_{\alpha}}X$ in $X$; it decomposes into eigenspaces of characters $a_{1,\alpha}, \ldots, a_{n,\alpha}$, defined up to permutation of indices (here we identify algebraic characters of $G$ with $\bC^\times$-valued characters). By the assumption that the action of $G$ is generically free, we have
$$
\sum_{i}  \bZ a_{i,\alpha} = A.
$$
This does not depend on the choice of $x_{\alpha}\in F_{\alpha}$. 
The dimension $\dim(F_{\alpha})$ equals the number of 
zeros among the $a_{i,\alpha}$. 
Thus we have a symbol, for each $\alpha$,
$$
[a_{1,\alpha}, \ldots, a_{n,\alpha}] \in \mathcal B_n(G).
$$
Put
\begin{equation}
\label{eqn:main} 
\beta(X) := \sum_{\alpha} \, [a_{1,\alpha}, \ldots, a_{n,\alpha}] 
\end{equation}
One of our main results is that expression \eqref{eqn:main}, considered as an element in $\mathcal B_n(G)$, 
is invariant under $G$-equivariant blowups. 

\begin{theo}
\label{thm:main}
The class $\beta(X) \in \mathcal B_n(G)$ is a $G$-equivariant birational invariant. 
\end{theo}

Now we introduce another $\bZ$-module
$$
\mathcal M_n(G),
$$
generated by symbols
$$
\langle a_1, \ldots, a_n\rangle, 
$$
such that $a_1,\ldots, a_n$ generate $A$, 
and subject to relations which are almost identical to those for $\mathcal B_n(G)$:
\begin{itemize}
\item[(S)] 
for all $\sigma \in \mathfrak S_n$ and all $a_1,\ldots, a_n\in A$ we have
$$
\langle a_{\sigma(1)} , \ldots, a_{\sigma(n)} \rangle  = \langle a_1,\ldots, a_n\rangle ,
$$ 
\item[(M)] 
for all $2\le k \le n$, all $a_1,\ldots, a_k\in A$ and all 
$b_1,\ldots, b_{n-k} \in A$, such that 
$$
\sum_i \bZ a_i + \sum_j \bZ b_j = A,
$$
we have
$$
\langle a_1,\ldots ,a_k, b_1, \ldots, b_{n-k}\rangle  =
$$
$$
= \sum_{1\le i\le k}    
\langle a_1-a_i, \ldots , a_i (\text{on $i$-th place}),  \ldots, a_k-a_i, b_1, \ldots, b_{n-k}\rangle. 
$$
\end{itemize}
Note that we eliminated the constraint $a_i\neq a_{i'}$, for $i'<i$, from the sum. 
Clearly,  
$$
\mathcal M_1(G) = \begin{cases} \bZ^{\phi(N)} & \text{if } G=\bZ/N\bZ, N\ge 1 \\ 0 & \text{otherwise.}\end{cases}
$$
For $n=4$ and $k=3$ and $a_1=a_2=a$ and $a_3=a'\neq a$ and $b_1=b$, the relation (M) translates to 
\begin{equation}\label{eqn:n4k3M}
\langle a,a,a',b\rangle  = \langle a,0,a'-a,b\rangle  + \langle 0,a,a'-a,b\rangle  + \langle a-a',a-a',a',b\rangle.
\end{equation}
The right side equals to 
$$
2\langle a,0,a'-a,b\rangle  + \langle a-a',a-a',a',b\rangle,
$$
by symmetry relations. Notice the difference between \eqref{eqn:n4k3M} and \eqref{eqn:n4k3B}.

In Section~\ref{sect:hecke}, we show that relation (M) follows from the subcase $k=2$.

These groups carry naturally defined, commuting, linear operators 
$$
T_{\ell,r} : \mathcal M_n(G) \ra \mathcal M_n(G),
$$
for all primes $\ell$ coprime to the order of $G$ and all $1\le r\le n$. We call these {\em Hecke operators}. One can consider their spectrum 
for 
$$
\mathcal M_n(G)\otimes \bar{\bQ} \quad \text{ or } \quad \mathcal M_n(G)\otimes \bar{\mathbb F}_p,
$$
where  $p$ is any prime not dividing $\# G$, the order of the group $G$. 
We expect that the joint spectrum of $T_{\ell,r}$ is related to automorphic forms and present evidence for this in Sections~\ref{sect:alge} and \ref{sect:modular}. 

Consider the map 
$$
\mu: \mathcal B_n(G)\ra \mathcal M_n(G)
$$
defined on symbols as follows:
\begin{itemize}
\item[($\bm{\mu}_0$)]  \quad 
$[a_1,\ldots, a_n] \mapsto \langle a_1,\ldots, a_n\rangle,
\quad  \text{ if all $a_1,\ldots, a_n \neq 0$},$
\item[($\bm\mu_1$)] \quad 
$
[0, a_2,\ldots, a_n] \mapsto 2\langle 0, a_2,\ldots, a_n\rangle, \quad \text{if all $a_2,\ldots, a_n \neq 0$},
$
\item[($\bm\mu_2$)] 
 \quad $ 
[0,0,a_3,\ldots, a_n] \mapsto 0, \quad 
\text{ for all $a_3,\ldots, a_n$}, 
$
\end{itemize}
and extended by $\bZ$-linearity.

\begin{theo}
\label{thm:main2}
The map $\mu$ is a well-defined homomorphism, which is a surjection modulo 2-torsion. 
\end{theo}

Note that 
$$
\langle 0,0,a_3,\ldots, a_n\rangle =0 \in \mathcal M_n(G)
$$
which follows from the relations by putting 
$$
k=2, a_1=a_2=0, b_i=a_{i+2}, \text{ for all }  i=1,\ldots, n-2.
$$  

We expect that $\mu$ is an isomorphism, modulo torsion (see Conjectures~\ref{conj:test} and \ref{conj:order}). 

\

Our notation $\mathcal B_n(G)$ and  $\mathcal M_n(G)$ stands for 

\

\centerline{{\bf birational} vs. {\bf motivic/modular.}} 

\

This paper consists of two parts: in Part 1, we present proofs of Theorems~\ref{thm:main} and 
\ref{thm:main2}. We recast the definition of $\mathcal M_n(G)$ in terms of scissor type relations on lattices with cones. 
We introduce a certain quotient $\mathcal M_n^-(G)$ of $\mathcal M_n(G)$ and define
multiplication and co-multiplication on this group. We formulate a series of conjectures reducing the structure of $\mathcal M_n(G)\otimes \bQ$ 
to certain primitive pieces. 
We define Hecke operators on $\mathcal M_n(G)$, which are compatible with the hypothetical decomposition.

In Part 2, we introduce various generalizations of 
$\mathcal B_n(G)$ and  $\mathcal M_n(G)$, not necessarily related to each other, reflecting a certain divergence of
birational and automorphic sides. Our considerations led us to a new question (see Question~\ref{ques:maybe} in Section ~\ref{sect:alge}), and a potentially new viewpoint on the Langlands program, based 
on higher-dimensional generalizations of modular symbols.  
We  identify  $\mathcal M_n^-(G)$  with cohomology of an arithmetic group, with coefficients in a 1-dimensional representation. 
We also explore, in the case $n=2$, the relation between our groups of symbols and classical Manin symbols for modular forms of weight 2.

During the preparation of this paper we discovered the work of Borisov-Gunnels \cite{borisov}, 
who studied constructions related to the modular picture in the case $n=2$ and raised the question of generalizations
to $n\ge 3$ in \cite[Remark 7.15]{borisov-higher}.

In the last section, we present results of computer experiments with equations for new invariants.

\

\noindent
{\bf Acknowledgments:}  
The research of V.P.  received funding from the European Research Council (ERC) under the European Union Horizon 2020 research and innovation program (QUASIFT grant agreement 677368). The research of Y.T. received support from  NSF grant 1601912.

We are grateful to Alex Barnett and Nick Carriero (Flatiron Institute, Simons Foundation) for their help with computer experiments, and to 
Avner Ash and 
Alexander Goncharov for their interest and helpful comments.

\centerline{\bf Part 1}

 \section{Invariance under blowups}
 \label{sect:relation}
 
We use notation and conventions from the Introduction. Let $X$ be  a smooth irreducible projective $n$-dimensional variety equipped with generically free regular action of  a finite abelian group $G$, and $W\subset X$ a closed smooth irreducible $G$-stable subvariety, 
 $$
0\le  \dim(W) \le n-2.
$$
Let 
$$
\pi: \tilde{X} = \Bl_W(X)\ra X
$$
be the blowup of $X$ in $W$. By the $G$-equivariant Weak Factorization theorem, smooth projective $G$-birational models of $X$ are connected by 
iterated blowups of such type.

In order to prove Theorem \ref{thm:main}, it suffices to show that 
$$
\beta(\tilde{X}) = \beta(X) \in \mathcal B_n(G).
$$  
Choose an irreducible component $Z\subset W^G$. 
It suffices to consider the structure of the fixed locus of exceptional divisors in 
the neighborhood of $Z$. 
Let 
$$
F=F(Z)\subseteq X^G
$$ 
be the unique irreducible component containing $Z$, it equals one of 
the $F_{\alpha}$ in \eqref{eqn:dec}. Let $z\in Z$ be a point and 
 $$
 \mathcal T_z X = T_1\oplus T_2\oplus R_1\oplus R_2
 $$
 the decomposition of the tangent bundle at $z$, where $T_i$ stand for trivial representations, 
 and $R_1$, $R_2$ have only nontrivial characters, with 
 $$
 \mathcal T_z X^{G} = T_zF =T_1\oplus T_2, \quad \mathcal T_z W = T_2\oplus R_1.
 $$
 Let 
 $$
 d_1:=\dim(T_1), \quad d_2 = \dim(T_2), \quad d_3= \dim(R_1), \quad d_4=\dim(R_2).
 $$
 The spectrum of the action of $G$ in $\mathcal T_z$ takes the form
$$
 \underbrace{0,\ldots, 0}_{d_1}\mid \underbrace{0,\ldots, 0}_{d_2} \mid  b_1,\ldots, b_{d_3} \mid  
 \underbrace{a^1,\ldots, a^1}_{\kappa_1}, \ldots, \underbrace{a^m,\ldots, a^m}_{\kappa_m},
 $$
 where $b_j\in A\setminus 0$, and $a^1, \ldots , a^m\in A\setminus 0$, pairwise distinct, with 
 $$
 \kappa_1+ \cdots + \kappa_m = d_4, \quad \kappa_i\ge 1, m\ge 0. 
 $$
 We have
 \begin{itemize}
 \item $d_2=\dim(Z)$,
 \item $d_1+d_2+d_3+d_4 =n$,
 \item $1\le d_3+d_4$, since $\codim(X^G)\ge 1$,
 \item $2\le d_1+d_4$, since $\codim(W)\ge 2$.
  \end{itemize}
 We consider cases, with corresponding 
 geometric configurations:
 \begin{itemize}
  \item[(I)]
  $d_1=0, d_4\ge 2$,
  geometrically, this means that $W$ contains a component of $X^G$. Blowing up $W$ we obtain new 
  contributions to formula \eqref{eqn:main}. 
The new fixed locus, with $m$ irreducible components, 
consists of subvarieties of the exceptional divisor, a projective bundle over $W$. These subvarieties, in turn, are total spaces of projective bundles over $Z$, with fibers 
  $$
  \bP^{\kappa_i-1}, \quad i=1, \ldots, m.
  $$ 
  The corresponding contribution to $\beta(\tilde{X})$ is given by
  $$
  \sum_{i=1}^m [\underbrace{0,\ldots}_{d_2}, b_1,\ldots, b_{d_3} , \underbrace{a^1-a^i, \ldots }_{\kappa_1}, \ldots, a_i,
  \underbrace{0, \ldots}_{\kappa_{i}-1}, \ldots, 
  \underbrace{a^m-a^i, \ldots}_{\kappa_m}].
  $$
 Putting 
 $$
 a_1,\ldots, a_k = \underbrace{a^1, \ldots}_{\kappa_1}, \ldots,
  \underbrace{a^m, \ldots }_{\kappa_m}
  $$
  and 
  $$
  b_1,\ldots, b_{n-k} = b_1, \ldots, b_{d_3}, \underbrace{0, \ldots}_{d_2}
  $$
we find that the formula matches relation (B), in the case when the sequence $\bar{a}=a_1,\ldots, a_k$ does not contain zeros. 
  \item[(II)]
  $d_1, d_4\ge 1$, 
  geometrically, this means that the tangent spaces of the fixed locus and $W$ do not span the whole tangent space and, near $Z$,
 the component $F$ is not contained in $W$. In the blowup, we will have a component of the fixed locus which is birational to $F$ and new components which are projective bundles 
$$
\bP^{\kappa_1-1}, \ldots,  \bP^{\kappa_m-1}
$$
over $Z$. We need to show that the contribution of these $m$ terms vanishes in $\mathcal B_n(G)$.  
The new components contribute
  $$
  \sum_{i=1}^m [\underbrace{-a^i,\ldots}_{d_1}, \underbrace{a^1-a^i, \ldots }_{\kappa_1}, \ldots, a_i,
  \underbrace{0, \ldots}_{\kappa_{i}-1}, \ldots, 
  \underbrace{a^m-a^i, \ldots}_{\kappa_m}, \bar{b}].
$$
We claim that this sum vanishes in $\mathcal B_n(G)$. Indeed, consider relation (B) for the sequences 
$$
\bar{a} =a_1,\ldots, a_k = \underbrace{0, \ldots}_{d_1}, \ldots,
  \underbrace{a^1, \ldots }_{\kappa_1}, \ldots,  \underbrace{a^m, \ldots }_{\kappa_m},
  $$
  and, as before, 
  $$
\bar{b} =   b_1,\ldots, b_{n-k} = b_1, \ldots, b_{d_3}, \underbrace{0, \ldots}_{d_2}.
  $$
The left side of (B) equals
$$[\bar{a},\bar{b}]=
[a_1,\ldots, a_k, \bar{b}] = [\underbrace{0, \ldots}_{d_1}, 
  \underbrace{a^1, \ldots }_{\kappa_1}, \ldots,  \underbrace{a^m, \ldots }_{\kappa_m}, \bar{b}]
$$
The right side is the sum of $(m+1)$ terms. The first summand, corresponding to $a_i=a_1=0$ coincides with the left side. 
The remaining terms are the same as above. 
 
\item[(III)]
  $d_1\ge 2,d_3\ge 1, d_4=0$, in this case, $F$ is not contained in $W$, no new contributors to formula \eqref{eqn:main} arise. 
\end{itemize}
 
 This concludes the proof of Theorem~\ref{thm:main}.

\section{Comparison}
 \label{sect:comp}

In this section we study the map
\begin{equation}
\label{eqn:mu12}
\mu: \mathcal B_n(G)\ra \mathcal M_n(G)
\end{equation} 
defined in Section~\ref{sect:intro}. The proof that this is a well-defined homomorphism is a long chain
of essentially trivial steps. 

First we record several corollaries of defining relations for $\mathcal M_n(G)$:
\begin{enumerate}
\item[(1)] $\langle 0,0, \ldots\rangle = 0$,
\item[(2)] $\langle a,a, \ldots\rangle = 2 \langle a,0, \ldots\rangle$,
\item[(3)] $\langle a,a,0, \ldots\rangle = 0$,
\item[(4)] $\langle a,a,a',a', \ldots\rangle = 0$,
\item[(5)]  $\langle a,a,a, \ldots\rangle = 0$,
\item[(6)]  $\langle a,-a,\ldots\rangle = 0$,
\end{enumerate}
here $\ldots$ stands for arbitrary sequences of elements in $A$, such that the set of all elements of the symbol spans 
the whole $A$.

In the proofs below we freely use the symmetry relation (S). 
\begin{itemize}
\item[(1)] 
We use (M) for $k=2$ and $a_1=a_2=0$:
$$
\langle 0,0, \ldots\rangle  = \langle 0, 0, \ldots \rangle  + \langle 0, 0, \ldots \rangle.
$$
\item[(2)] 
We use (M) for $k=2, a_1=a_2=a$.
\item[(3)] 
We use (2) and (1):
$$
\langle a,a,0, \ldots\rangle  \stackrel{(2)}{=}  2\langle a, 0, 0, \ldots \rangle  + \langle 0, 0, \ldots \rangle \stackrel{(1)}{=}0. 
$$
\item[(4)]
We   use again (2) and (1):
$$
\langle a,a,a',a', \ldots\rangle  \stackrel{(2)}{=} 4 \langle a, 0, a',0, \ldots \rangle  \stackrel{(1)}{=} 0.
$$
\item[(5)] We use (M) for $k=3$ and $a_1=a_2=a_3=a$, and then (1):
$$
\langle a,a,a, \ldots\rangle = 3  \langle a, 0, 0, \ldots \rangle \stackrel{(1)}{=} 0,
$$
\item[(6)] 
We use (M) for $k=2, a_1=a, a_2=0$:
$$
\langle a,0, \ldots\rangle = \langle a,-a,\ldots\rangle + \langle a,0 \ldots\rangle.
$$
\end{itemize}

We proceed to the proof of Theorem~\ref{thm:main2}. The main point is to check the following compatibility equation
\begin{equation}
\label{eqn:mu1}
\mu([a_1,\ldots, a_k, b_1,\ldots, b_{n-k}]) = 
$$
$$
\sum_{i,  a_i\neq a_{i'}, \text{ for } i<i'} \mu([ a_1-a_i, \ldots, a_i, \ldots, a_k-a_i, b_1, \ldots, b_{n-k}]).
\
\end{equation}
For convenience, we sometimes write
$$
[a_1,\ldots, a_k \mid  b_1,\ldots, b_{n-k}] = [a_1,\ldots, a_k, b_1,\ldots, b_{n-k}] \in \mathcal B_n(G),
$$
and similarly, for the symbol in $\mathcal M_n(G)$,  
indicating the position of the separation of $a$ and $b$ variables in subsequent relations. 

There are three cases, distinguished by the number of zeros in the sequence
$$
\bar{b}:= b_1,\ldots, b_{n-k},
$$
\begin{itemize}
\item[(C0)] $\bar{b}$ does not contain zeros. 
\item[(C1)] $\bar{b}$ contains exactly one zero.  
\item[(C2)] $\bar{b}$ contains at least two zeros. 
\end{itemize}

The case (C2) is obvious, by relation (1), since all terms vanish, by the definition $(\bm\mu_2)$ (in Section~\ref{sect:intro}).

The case (C1) splits into subcases
\begin{itemize}
\item[(C10)]
The sequence 
$$
\bar{a}:=a_1,\ldots, a_k
$$
contains no zeros,
\item[(C11)]
$\bar{a}$ contains at least one zero. 
\end{itemize}
In the case (C11), the left hand side maps to 0, by $(\bm\mu_2)$:
$$
\mu([0, \ldots \mid 0,\ldots]) = 0.
$$
The terms of the right hand side in the relation (B) are of two types, corresponding to $a_i=0$ or $a_i=a\neq 0$. 
If $a_i=0$, then the term has the form
$$
[\underline{0}, \ldots \mid 0, \ldots ], 
$$
mapping to zero, by $(\bm\mu_2)$.  The underlined $0$ indicates that $a_i$ is left in its place, in the relation (B). 
If $a_i=a\neq 0$, then the corresponding term in the right hand side of (B) has the form
$$
[-a, \cdots, \underline{a}, \ldots \mid 0, \ldots ], 
$$
mapping to 
$$
c\cdot \langle -a,\ldots, a, \ldots  0, \ldots \rangle, 
$$
where $c=0$ or $2$, and the symbol in $\mathcal M_n(G)$ equals 0, by (6). 

The case (C10) splits into two cases:
\begin{itemize}
\item[(C10$\neq$)]
all terms in $\bar{a}$ are pairwise distinct,
\item[(C10$=$)]
there exists at least two equal terms in $\bar{a}$. 
\end{itemize}
In case (C10$\neq$), in the left and in the right hand side of the relation (B), all symbols contain exactly one zero. 
Thus, they are mapped to similar symbols in $\mathcal M_n(G)$, but multiplied by 2, by $(\bm\mu_1)$. 
Since every element in $\bar{a}$ occurs only once, the expressions on the right side of (B) and (M) consist of matching terms. 

In case (C10$=$), the left hand side of (B) equals
$$
[a,a, \ldots \mid 0, \ldots ] \in \mathcal B_n(G).
$$
Its image under $\mu$ equals
$$
2 \langle a,a, \ldots, 0, \ldots \rangle \in \mathcal M_n(G),
$$
which vanishes, by (3). We claim that all terms on the right side of (B) map to zero as well. Indeed, they are
either of the form
$$
[ \underline{a}, 0, \ldots \mid 0, \ldots ]
\quad \text{ or } \quad 
[a-a', a-a', \ldots,  \underline{a}', \ldots \mid 0, \ldots ], \quad a'\neq a.
$$
The image of this symbol is proportional to 
$$
\langle a, 0, \ldots, 0, \ldots \rangle \quad \text{ or} \quad \langle a-a', a-a', \ldots, a', \ldots, 0, \ldots \rangle,
$$
vanishing by (1) or (3), respectively. 

The case (C0) splits into three cases:
\begin{itemize}
\item[(C00)] $\bar{a}$ does not contain zeros,
\item[(C01)] $\bar{a}$ contains exactly one zero,
\item[(C02)] $\bar{a}$ contains at least two zeros.
\end{itemize}
Recall that $\bar{b}$ does not contain zeros, in case (C0). 
We start with (C02). The left hand side in (B) has the form
$$
[0,0, \ldots \mid \ldots ],
$$
hence maps to 0, by $(\bm\mu_2)$. We check that all terms on the right hand side of (B) map to 0 as well. These symbols
have the form
$$
[\underline{0}, 0, \ldots \mid \ldots ] \quad \text{ or } \quad [-a, -a, \ldots, \underline{a}, \ldots \mid \ldots ], \quad a\neq 0, 
$$
mapping to elements in $\mathcal M_n(G)$ which are proportional to either
$$
\langle 0, 0, \ldots \rangle \quad \text{ or } \quad \langle -a, -a, \ldots, a, \dots \rangle, 
$$
vanishing by (1) or (6), respectively. 

The case (C01) splits into two cases: 
\begin{itemize}
\item[(C01$\neq$)]
all terms in $\bar{a}$ are pairwise distinct,
\item[(C01$=$)]
there exists at least two equal terms in $\bar{a}$. 
\end{itemize}
In case (C01$=$), the left side in (B) has the form
$$
[0, a,a, \ldots \mid \ldots ], \quad \text{ for } a\neq 0,
$$
mapping to 0, by relation (3). The right side contains terms of the form
$$
[ \underline{0}, a,a, \ldots \mid \ldots ]
\quad \text{ or } \quad 
[-a,\underline{a}, 0,  \ldots \mid \ldots ],
$$
or 
$$
[-a',a-a', a-a', \ldots, \underline{a}', \ldots  \mid \ldots ], 
\quad a'\neq a,0.
$$
Their images under $\mu$ are proportional to 
$$
\langle 0, a,a, \ldots \rangle, \quad \text{ or } \quad \langle -a, -a,0, \ldots \rangle, 
$$
or 
$$
\langle -a', a-a',a-a',\ldots, a', \ldots \rangle, 
$$
which vanish by (3), (6), and (6), respectively. 

Consider the case (C01$\neq$). The left side of (B) has the form
$$
[0, a_2, \ldots, a_k  \mid \ldots ], \quad \text{ for } a_i\neq 0, i\ge 2, \text{ pairwise distinct}, b_j\neq 0.
$$
Its image under $\mu$ equals, by $(\bm\mu_1)$, to 
$$
2 \langle 0, a_2, \ldots, a_k, \ldots \rangle.
$$
The right side of (B) is the sum
$$
[\underline{0}, a_2,  \ldots , a_k \mid \ldots ] + [-a_2,\underline{a_2}, \ldots,  a_k-a_2 \mid \ldots ] + 
[-a_3, a_2-a_3,\underline{a_3}, \ldots \mid \ldots ] + \cdots
$$
 where the first summand maps,  by $(\bm\mu_1)$, to 
 $$
 2 \langle 0, a_2, \ldots, a_k, \ldots \rangle
 $$
 and all the other terms map to 0, by relation (6). This proves (C01$\neq$).

We are left with the case (C00), i.e., all elements of the sequences $\bar{a}$ and $\bar{b}$ are nonzero. 
We have two cases:
\begin{itemize}
\item[(C00$\neq$)] 
all terms in $\bar{a}$ are pairwise distinct,
\item[(C00$=$)] at least two terms in $\bar{a}$ are equal. 
\end{itemize}
In case (C00$\neq$), the left and the right side of (B) do not contain symbols with zeroes, hence we use $(\bm\mu_0)$ and 
the relation (B) is mapped precisely to the corresponding relation (M).  

The case (C00$=$) splits into three subcases:
\begin{itemize}
\item[(C00$=2$)] $\bar{a}$ has only one pair of equal terms, i.e., 
$$
\bar{a} = a,a,a_3,\ldots, a_k, 
$$
where $a_3,\ldots, a_k $ are pairwise 
distinct and different from $a$,
 \item[(C00$=2,2$)] $\bar{a}$ has the form
 $$
 \bar{a} = a,a,a',a', a_5,\ldots, a_k, 
$$
where $a\neq a'$ and $a_5,\ldots, a_k $ are pairwise distinct and different from $a,a'$,
 \item[(C00$=3$)] $\bar{a}$ has the form
 $$
 \bar{a} =  a,a,a,\ldots 
 $$
 \end{itemize}
We start with (C00$=3$). The left side is mapped to 0, by relation (5). The right side has terms of the form
$$
[ \underline{a}, 0,0, \ldots \mid \ldots ]
\quad \text{ or } \quad 
[a-a',a-a',a-a', \ldots, \underline{a}', \ldots \mid \ldots ], \quad a\neq a'.
$$
They are mapped to terms proportional to 
$$
\langle a, 0, 0, \ldots \rangle \quad \text{ or} \quad \langle a-a', a-a',a-a', \ldots \rangle,
$$
vanishing by (1) or (5), respectively. 

We consider (C00$=2,2$). The left side is mapped to 
$$
\langle a,a,a',a' , \ldots\rangle
$$
which vanishes by relation (4). The right side has terms of three shapes
$$
[a,0,a'-a,a'-a, \ldots \mid \ldots ] 
\quad \text{ or } \quad 
[a-a',a-a',\underline{a}', 0, \ldots \mid \ldots ],  \quad a\neq a'.
$$
or 
$$
[a-a'',a-a'',a'-a'', a'-a'',\ldots, \underline{a}'', \ldots \mid \ldots ], \quad a, a', a'' \text{ pairwise distinct}. 
$$
Their images are proportional to 
$$
\langle a,0,a'-a,a'-a, \ldots\rangle
\quad \text{ or } \quad 
\langle a-a',a-a',\underline{a}', 0, \ldots \rangle,  \quad a\neq a'.
$$
or 
$$
\langle a-a'',a-a'',a'-a'', a'-a'',\ldots, \underline{a}'', \ldots\rangle, \quad a, a', a'' \text{ pairwise distinct},
$$
which vanish by (3), (3), and (4), respectively. 

In the last case (C00$=2$), relation (B) has the form
$$
[a,a,a_3,\ldots, a_k \mid \ldots ]  = 
[\underline{a},0,a_3-a,\ldots, a_k-a \mid \ldots ] +
$$
$$
+ [a-a_3,a-a_3,\underline{a}_3,\ldots, a_k-a_3 \mid \ldots ] + [a-a_4,a-a_4,a_3-a_4, \underline{a}_4,\ldots \mid \ldots ] + \cdots
$$
The left side maps to 
 $$
 \langle a,a,a_3,\ldots \rangle
 $$
 and the right side to 
 $$
 2  \langle \underline{a},0,a_3-a,\ldots, a_k-a \mid \ldots \rangle + \langle a-a_3,a-a_3,\underline{a}_3,\ldots, a_k-a_3\mid \ldots \rangle + \cdots.
 $$
 Here the first summand is obtained by $(\bm\mu_1)$ and the other summands by $(\bm\mu_0)$. 
 We see that, modulo relation (S), the image of the right hand side of (B) coincides with the right hand side of the relation (M) 
 in $\mathcal M_n(G)$. 
 
 This concludes the proof of Theorem~\ref{thm:main2}.

\begin{prop} \label{prop:n2comparison}
The homomorphism 
\begin{equation}
\label{eqn:mu2}
\mu: \mathcal B_2(G)\ra \mathcal M_2(G)
\end{equation} 
is injective, with 
cokernel isomorphic to $(\bZ/2\bZ)^{\phi(N)}$, if $G\simeq \bZ/N\bZ$ is a cyclic group, 
and is an isomorphism otherwise.
\end{prop}

\begin{proof}
One can write the generators and relations for $\mathcal B_2(G)$ and $ \mathcal M_2(G)$ as follows:
\begin{itemize}
\item{\bf Generators}: 
\begin{itemize}
\item
(``non-degenerate") symbols $[a_1,a_2]$ (resp.,  $\langle a_1,a_2\rangle$),  where $a_1,a_2 \in A\setminus 0$ are such that  $\bZ a_1+\bZ a_2=A$, and 
\item 
(``degenerate") symbols $[a,0]$ (resp., $\langle a,0\rangle$), where
 $a\in A\setminus 0$ is such that $\bZ a=A$,
 \end{itemize}
 \item{\bf Relations:}
 \begin{itemize}
\item[({\bf 1})] 
$[a_1,a_2]=[a_2,a_1]$ (resp. $\langle a_1, a_2\rangle=\langle a_2,a_1\rangle$) for $a_1,a_2\in A\setminus 0$,
\item[({\bf 2})] 
$[a_1,a_2]=[a_1,a_2-a_1]+[a_1-a_2,a_2]$ (and correspondingly $\langle a_1,a_2\rangle=\langle a_1,a_2-a_1\rangle +\langle a_1-a_2,a_2\rangle$) for  $a_1,a_2\in A\setminus 0$
and $a_1\ne a_2$,
\item[({\bf 3})] 
$[a,a]=[a,0]$ (resp. $\langle a,a\rangle=2\langle a,0\rangle$) for $a\ne 0$.
\end{itemize}
\end{itemize}
We see that the first two relations are identical and deal only with non-degenerate symbols $[a_1,a_2]$ (resp., $\langle a_1,a_2\rangle$), when both $a_1, a_2$ are nonzero. 
In the case $ \mathcal B_2(G)$, relation (3) just identifies the degenerate symbol
$[a,0]$ via the nondegenerate symbol $[a,a]$, whereas in the case of  $\mathcal M_2(G)$ it adds one half of the nondegenerate symbol $\langle a,a\rangle$. Obviously, if we add to any abelian group an extra generator which is one half of any given element of this group, then the new group contains the initial one, and the quotient is  $\bZ/2\bZ$.
The statement of the Proposition immediately follows from these considerations, as the Euler function $\phi(N)$ is the number of degenerate elements $[a,0]$ in the case $G\simeq A\simeq \bZ/N\bZ$.
\end{proof}

 \begin{conj}
 \label{conj:test}
For $n\ge 3$ the homomorphism
 \begin{equation}
\label{eqn:mu11}
\mu: \mathcal B_n(G)\ra \mathcal M_n(G)
\end{equation} 
 is an isomorphism, modulo torsion. 
 \end{conj}
 
 This statement reduces to the following: 
 For any integer $N\ge 2$,
  $$
 [0,0,1]\in \mathcal B_3(\bZ/N\bZ)
 $$
 is a torsion element. Indeed, if this were the case, then any symbol $[0,0,\ldots]$ would vanish modulo torsion, and then one could repeat all the steps in the proof of Theorem~\ref{thm:main2} and construct an inverse morphism from  $ \mathcal M_n(G)\otimes \bQ$ to  $ \mathcal B_n(G)\otimes\bQ$.

\
  
Computer experiments for $N\le 23$ support the following:
\begin{conj}
 \label{conj:order}
For $N\ge 2$, the element
$$
 [0,0,1]\in \mathcal B_3(\bZ/N\bZ)
 $$
has order 1, i.e., $[0,0,1]=0\in  \mathcal B_3(\bZ/N\bZ)$, if $N$ is composite or $N=2,3,5$, 
and order exactly equal to 
$$
\frac{p^2-1}{24}, \quad \text{  if } N=p\ge 7 \quad \text{ is a prime.}
$$
 \end{conj}

\section{On generators and relations in $\mathcal M_n(G)$}
\label{sect:gen-rel}

In this section, $G$ is a finite abelian group, with character group  $A=\Hom(G,\mathbb{C}^\times)$, and  $n\ge 2$ is an integer.  
 We give a geometric reformulation of generators and relations of $\mathcal M_n(G)$.

 We start with the following data:
 \begin{itemize}
 \item 
 a (torsion-free) lattice $\mathbf L\simeq \bZ^n$ of rank $n$,
 \item 
 an element $\chi\in \mathbf L \otimes A$ such that 
the induced homomorphism
$$
\mathbf L^\vee \ra A
$$
is a surjection,
\item 
a {\em basic simplicial cone}, i.e., a strictly convex cone 
$$
\Lambda\in \mathbf L_{\bR}
$$
spanned by a basis of $\mathbf L$. It is isomorphic to the standard octant $\bR^n_{\ge 0}$, for $\mathbf L=\bZ^n\subset \bR^n$. 
 \end{itemize}
 
For every equivalence class of triples
$$
(\mathbf L, \chi, \Lambda),
$$
up to isomorphism, we define a symbol
$$
\psi(\mathbf L, \chi, \Lambda) \in  \mathcal M_n(G)
$$
as follows: 
choose a basis $e_1,\ldots, e_n$ of $\mathbf L$, spanning $\Lambda$, and 
express 
\begin{equation}
\label{eqn:chi}
\chi=\sum_{i=1}^n 
e_i\otimes a_i,
\end{equation}
and put
$$
\psi(\mathbf L, \chi, \Lambda) = \langle  a_1,\ldots, a_n\rangle \in \mathcal M_n(G).
$$
The ambiguity in the choices is reflected in the action of the symmetric group $\mathfrak S_n$ on the basis elements, 
hence accounted for by condition (S). Relation (M) has the following geometric meaning: let 
$e_1,\ldots, e_n$ be an ordered basis of $\mathbf L$ spanning $\Lambda$:
\begin{equation}
\Lambda:= \bR_{\ge 0} e_1 +  \cdots + \bR_{\ge 0} e_n.
\end{equation}
 Fix an integer $2\le k\le n$. 
Then  
\begin{equation}
\label{eqn:lam}
\Lambda = \Lambda_1\cup \cdots \cup \Lambda_k,
\end{equation}
where 
$$
\Lambda_i:= \bR_{\ge 0} e_1 +  \cdots + \underbrace{\bR_{\ge 0} (e_1+\cdots e_k)}_{\text{$i$-th place}} + \cdots + \bR_{\ge 0} e_n,
$$
i.e., we are replacing the $i$-th generator $e_i$ by $(e_1+ \cdots e_k)$. 
The cones $\Lambda_i$ are also basic simplicial and their 
their interiors are disjoint. Decompose 
$$
\chi = e_1\otimes a_1 + \cdots + e_k\otimes a_k  + e_{k+1} \otimes b_{1} + \cdots + e_{n} \otimes b_{n-k} 
$$
as in \eqref{eqn:chi}, i.e., $a_{k+i}=b_{i}$, for all $i=1, \ldots, n-k$. Then, in the basis of $\Lambda_i$, $\chi$
decomposes as
$$
e_1\otimes (a_1-a_i) + \cdots + (e_1+\cdots e_k)\otimes a_i+ 
\cdots 
e_k\otimes (a_k-a_i)  + \sum_{j=1}^{n-k} e_{k+j} \otimes b_{j}.
$$
We see that relation (M) can be expressed as the following identity
\begin{equation}
\label{eqn:Chi}
\psi(\mathbf L, \chi, \Lambda) = \sum_{i=1}^k \psi(\mathbf L, \chi, \Lambda_i),
\end{equation}
which we can view as an analog of scissor relations. Our next result is that this relation follows from the special subcase $k=2$. This 
is a corollary of a general result concerning simplicial subdivisions of basic simplicial cones. Namely, consider the $\bZ$-module
$$
\mathcal F_{\mathbf{L}, \bZ}
$$
generated by symbols
$[\Lambda]$, 
where $\Lambda$ is a basic simplicial cone, modulo relations $(\mathrm{R}_k)$,  $k\ge 2$:

\begin{itemize}
\item
$
[\Lambda] = [\Lambda_1]+\cdots + [\Lambda_k],
$
\end{itemize} 
where 
$\Lambda$ and $\Lambda_i$ are as above, with $e_1,\ldots, e_n$ an arbitrary basis of $\Lambda$.

\begin{lemm} 
\label{lemm:chi}
Relations $(\mathrm{R}_k)$ for $k\ge 3$, follow from relations $(\mathrm{R}_2)$. 
\end{lemm}

\begin{proof}
 We proceed by induction, assuming the claim for $k-1$. 
 We want to prove the claim for $k\ge 3$, i.e., 
 $$
 [\Lambda]_1+\cdots  + [\Lambda_k] = [\Lambda].
 $$
 By induction, 
$$
[\Lambda_k] =[\Lambda'_1]+ \cdots + [\Lambda_{k-1}'],
$$
 where
 $$
 \Lambda_i':= \bR_{\ge 0} e_1 +  \cdots + \underbrace{\bR_{\ge 0} (e_1+\cdots +e_{k-1})}_{\text{$i$-th place}} + \cdots + 
 \underbrace{\bR_{\ge 0} (e_1+\cdots + e_{k})}_{\text{$k$-th place}} +\cdots 
  +\bR_{\ge 0} e_n,
$$ 
 indeed, this is the relation $(\mathrm{R}_{k-1})$ in the basis 
 $$
 e_1,\ldots, e_{k-1}, (e_1+\cdots + e_k), e_{k+1},\ldots e_n.
 $$
 Therefore, 
 $$
 [\Lambda_1]+ \cdots + [\Lambda_k] = ([\Lambda_1]+[\Lambda_1'] )+\cdots +  ([\Lambda_{k-1}]+[\Lambda_{k-1}']).
 $$
 For every $i=1,\ldots, k-1$, we have the relation $(\mathrm{R}_2)$
 $$
 [\Lambda_i]+[\Lambda_i'] = [\Lambda_i''],
 $$
in an appropriate basis, 
 where
 $$
 \Lambda_i'':= \bR_{\ge 0} e_1 +  \cdots + \underbrace{\bR_{\ge 0} (e_1+\cdots +e_{k-1})}_{\text{$i$-th place}}  +\cdots 
 + \bR_{\ge 0} e_n.
$$ 
 Finally, $(\mathrm{R}_{k-1})$ in the basis $e_1,\ldots, e_n$ says that
 $$
 [\Lambda_1'']+ \cdots + [\Lambda_{k-1}''] =[\Lambda],
 $$
 which proves the claim.
\end{proof}

Now we can consider an {\em a priori} different group generated by symbols $[\Lambda]$, where $\Lambda$ is any full-dimensional strictly convex rational polyhedral cone, subject to relations
$$
[\Lambda] =  [\Lambda]_1+\cdots  + [\Lambda_k],
$$
where $\Lambda$ is the union of cones $\Lambda_i$ with disjoint interiors (here $k$ can be any integer $\ge 2$. 
The toric analog of Weak Factorization implies that the natural homomorphism from $\mathcal F_{\mathbf{L}, \bZ}$ to this group is an isomorphism. 
In these terms, Lemma~\ref{lemm:chi} says that it suffices to consider blowups with centers in codimension 2.

In consequence, $\mathcal M_n(G)$ admits an alternative description: as the group generated by symbols
$$
\psi(\mathbf L, \chi, \Lambda),
$$
depending only on the isomorphism classes of triples, 
where $\mathbf L$ and $\chi$ are as above, and $\Lambda$ is a finitely generated convex rational polyhedral cone, of full dimension, subject to the relations \eqref{eqn:Chi}, whenever there is a decomposition
$$
\Lambda= \Lambda_1\cup \cdots \cup \Lambda_k
$$
as above. This clearly extends to nonconvex cones.

We introduce a variant of previous constructions: instead of 
 $$
 \chi\in \mathbf L\otimes A= \Hom(\mathbf L^\vee, A)
 $$ 
 we can consider 
 $$
 \chi^* \in  \Hom(\mathbf L, A),
 $$
 again assuming that $\chi^*$ is surjective. 
 In a similar fashion, we can introduce the group $\mathcal M^*_n(G)$, which we call the {\em co-vector}
 version of (the {\em vector} version) $\mathcal M_n(G)$. This group is 
generated by symbols,  
  $$
 \langle a_1,\ldots, a_n\rangle^*,
 $$
 subject to relations 
 \begin{itemize}
 \item[$(\mathrm{S}^*)$] for all 
 $\sigma\in \mathfrak S_n$ and all $a_1,\ldots, a_n\in A$ we have
$$
\langle a_{\sigma(1)} , \ldots, a_{\sigma(n)} \rangle^*  = \langle a_1,\ldots, a_n\rangle^*,
$$ 
\item[$(\mathrm{M}^*)$] 
for all $2\le k \le n$, all $a_1,\ldots, a_k\in A$ and all 
$b_1,\ldots, b_{n-k} \in A$ such that 
$$
\sum_i \bZ a_i + \sum_j \bZ b_j = A
$$
we have
$$
\langle a_1,\ldots ,a_k, b_1, \ldots b_{n-k}\rangle^*  =
$$
$$
= \sum_{1\le i\le k}    
\langle a_1, \ldots , \sum_{j=1}^{k} a_j (\text{on $i$-th place}),  \ldots, a_k, b_1, \ldots, b_{n-k}\rangle^* 
$$
\end{itemize} 

As above, the relations for $k=2$ imply all others. 

It is not hard to show  that the $\bQ$-ranks of $\mathcal M_n(G)$ and $\mathcal M_n^*(G)$ are the same.
Indeed, by M\"{o}bius-type inversion formula, one can reduce the question to the extended versions of groups $\mathcal M_n(G)$ and $\mathcal M_n^*(G)$ omitting the condition that the
map 
$$
\chi:\mathbf L^\vee \to A, \quad \text{resp. }\quad \chi^*:\mathbf L\to A,
$$ 
is surjective. Then the finite Fourier transform (after a choice of an identification $G\simeq A$) identifies two complex vector spaces consisting of homomorphisms from two extended groups to $\bC$.

\section{Multiplication and co-multiplication}
\label{sect:co}

In this section, we work with the vector version, the co-vector version is analogous. 
We consider 
$$
\mathcal M_n(G)
$$
in {\em both} variables $n\ge 1$ and $G$. We define multiplication and co-multiplication maps and study their properties. 
An important role will be played by 
$$
\mathcal M_n^-(G),
$$
which is
defined {\em only for nontrivial groups} $G$, as the quotient of $\mathcal M_n(G)$ by the relation
\begin{equation}
\label{eqn:anti}
\langle -a_1, \ldots, a_n\rangle =- \langle a_1, \ldots, a_n\rangle.
\end{equation}
We denote by 
$$
\langle a_1, \ldots, a_n\rangle^{-} \in \mathcal M_n^-(G)
$$
the image of $\langle a_1, \ldots, a_n\rangle$ under the natural projection
$$
\mu^-: \mathcal M_n(G) \ra \mathcal M_n^-(G).
$$

We consider short exact sequences of finite abelian groups
$$
0\ra G'\ra G\ra G''\ra 0
$$
and the corresponding short exact sequences of character groups
$$
0\ra A''\ra A\ra A'\ra 0.
$$
Let 
$$
n=n'+n'', \quad n',n''\ge 1.
$$

We define a $\bZ$-bilinear `multiplication' map
$$
\nabla: \mathcal M_{n'}(G') \otimes \mathcal M_{n''}(G'')\ra \mathcal M_{n'+n''}(G),
$$
which on generators is given 
by the formula
\begin{equation}
\label{eqn:mult-def}
\langle a_1',\ldots, a_{n'}'\rangle \otimes \langle a_1'',\ldots, a_{n''}''\rangle \mapsto \sum \langle a_1,\ldots , a_{n'} , a_1'',\ldots, a_{n''}''\rangle,
\end{equation}
where the sum runs over all lifts $a_i\in A$ of $a_i'\in A'$, and the elements $a''_i$ are understood as elements of $A$,  via the embedding $A''\hookrightarrow A$. 

The compatibility with defining relations (S) and (M) is obvious. The condition that the elements in each summand on the right span $A$ follows
from the corresponding condition on the left for the groups $A',A''$. 
Note that $\nabla$ descends to a $\bZ$-bilinear map of corresponding quotient groups
$$
\nabla^-: \mathcal M_{n'}^-(G') \otimes \mathcal M_{n''}^-(G'')\ra \mathcal M_{n'+n''}^- (G),
$$
where both $G'$ and $G''$ are nontrivial. 

Next we define a `co-multiplication' map 
$$
\Delta: \mathcal M_{n'+n''}(G)\ra \mathcal M_{n'}(G') \otimes \mathcal M_{n''}^- (G''),
$$
where $G''$ is nontrivial, and  
which on generators is given 
by the formula
\begin{equation}
\label{eqn:first}
\langle a_1,\ldots, a_{n}\rangle \mapsto  \sum \,\,\, \langle a_{I'} \text{ mod } A'' \rangle \otimes  \langle a_{I''}\rangle^-.
\end{equation}
Here we put
$$
\langle a_{I'} \text{ mod } A'' \rangle=\langle a_{i_1} \text{ mod } A'', \ldots, a_{i_{n'}} \text{ mod } A'' \rangle, \quad I':=\{ i_1,\ldots, i_{n'}\}
$$
and, similarly, for $\langle a_{I''}\rangle$, using the symmetry relation (S). 
The sum is over all subdivisions 
$$
\{ 1,\ldots , n\} = I'\sqcup I'', \quad \text{  with } \# I'=n', \#I'' = n'',
$$
such that 
\begin{itemize}
\item 
for all $j\in I''$, we have $a_j\in A''\subset A$, and, in the first term on the right, 
the elements  $a_i, i\in I'$, are replaced by their images in  $A'=A/A''$;
\item 
(generation condition) the elements $a_j, j\in I'',$ span $A''$.
\end{itemize}
Note that, given the generation condition in each term of the right side of the formula, the expression
$\langle a_{I'} \text{ mod } A'' \rangle^{-}$ is a symbol, since the condition 
$\sum \bZ a_i =A$ implies that $\sum_{i\in I'} (a_i  \text{ mod } A'') =A'$. 
Therefore, the generation condition for the first term is automatic. 

\begin{prop}
\label{prop:delta}
The map $\Delta$ extends to a well-defined $\bZ$-linear homomorphism. 
\end{prop}

\begin{proof}
For any $a\in A$, we put
$$
\delta^{gen}_{a\in A''}:=\begin{cases} 1 & a\in A'' \text{ and }  \bZ a+\sum_{j\in J''} \bZ a_j = A'', \\ 0 & \text{ otherwise. } \end{cases} 
$$
By Lemma~\ref{lemm:chi}, it suffices to check 2-term relations ($\mathrm{R}_2$).
We need to show that the image of the relation 
$$
\langle a_1,a_2,\ldots \rangle = \langle a_1-a_2, \ldots\rangle + \langle a_1,a_2-a_1, \ldots \rangle 
$$
on the left is a relation on the right, and that the terms on the right satisfy the generation condition 
(linear combinations of elements span the corresponding group). 
The only interesting part is when the first two arguments are distributed over the different factors in \eqref{eqn:first}, so that 
\begin{multline}
\langle a_1,a_2,\ldots \rangle  \mapsto \delta^{gen}_{a_1\in A''} \cdot \langle a_2 \text{ mod } A'', \ldots\rangle \otimes \langle a_1, \ldots \rangle ^-   \\
+ \delta^{gen}_{a_2\in A''} \cdot \langle a_2 \text{ mod } A'', \ldots\rangle \otimes \langle a_2, \ldots \rangle ^- 
\end{multline}
There are four cases: 
\begin{enumerate}
\item $a_1\in A'', a_2 \in A''$
\item $a_1\in A'', a_2\notin A''$
\item $a_1\notin A'', a_2\in A''$
\item $a_1\notin A'', a_2\notin A''$
\end{enumerate}
We fix disjoint subsets
$$
J':=I'\cap \{ 3, \ldots, n\} , \quad J'':=I''\cap \{ 3, \ldots,n\}
$$ 
of cardinality $n'-1$, respectively, $n''-1$. 
For each symbol on the left of \eqref{eqn:first} there are at most two nonzero terms on the right (depending the generation condition)
corresponding to the cases $a_1\in I', a_2\in I''$ or
$a_1\in I'', a_2\in I'$.

In Case (1), we have
$$
\langle a_1,a_2,\ldots \rangle  \mapsto \delta^{gen}_{a_1\in A''} \cdot \langle 0 , \ldots\rangle \otimes \langle a_1, \ldots \rangle ^-   
+ \delta^{gen}_{a_2\in A''} \cdot \langle 0 , \ldots\rangle \otimes \langle a_2, \ldots \rangle ^- 
$$
and 
\begin{multline*}
\langle a_1-a_2,a_2,\ldots \rangle 
+ \langle a_1,a_2-a_1,\ldots \rangle 
\mapsto 
\\
  \delta^{gen}_{a_1-a_2\in A''}  \cdot \langle 0, \ldots \rangle \otimes \langle a_1-a_2, \ldots\rangle^-  \\
   +\delta^{gen}_{a_2\in A''}  \cdot \langle 0 , \ldots \rangle \otimes \langle a_2, \ldots\rangle^- 
    +\delta^{gen}_{a_1\in A''}  \cdot \langle 0 , \ldots \rangle \otimes \langle a_1, \ldots\rangle^- \\
  +\delta^{gen}_{a_2-a_1\in A''}  \cdot \langle 0 , \ldots \rangle \otimes \langle a_2-a_1, \ldots\rangle^- \\
\end{multline*}
The first and the last term on the right cancel by relation \eqref{eqn:anti}, and the second and the third term are the image of 
$\langle a_1,a_2,\ldots \rangle$.

In Case (2), we have
$$
\langle a_1,a_2,\ldots \rangle  \mapsto \delta^{gen}_{a_1\in A''} \cdot \langle a_2 \text{ mod } A'' , \ldots\rangle \otimes \langle a_1, \ldots \rangle ^-   
$$
and 
\begin{multline*}
\langle a_1-a_2,a_2,\ldots \rangle 
+ \langle a_1,a_2-a_1,\ldots \rangle 
\mapsto \\
  \delta^{gen}_{a_1\in A''}  \cdot \langle a_2-a_1 \text{ mod } A'' , \ldots \rangle \otimes \langle a_1-a_2, \ldots\rangle^-  
 \end{multline*}
 The right sides of both expressions coincide, since $a_2=a_2-a_1 \text{ mod } A''$. 
Case (3) is similar to Case (2).

In Case (4), we have
$$
\langle a_1,a_2,\ldots \rangle  \mapsto  0
$$
and 
\begin{multline*}
\langle a_1-a_2,a_2,\ldots \rangle 
+ \langle a_1,a_2-a_1,\ldots \rangle 
\mapsto \\
  \delta^{gen}_{a_1-a_2\in A''}  \cdot \langle a_2 \text{ mod } A'' , \ldots \rangle \otimes \langle a_1-a_2, \ldots\rangle^-  
  \\  + 
    \delta^{gen}_{a_2-a_1\in A''}  \cdot \langle a_1 \text{ mod } A'' , \ldots \rangle \otimes \langle a_2-a_1, \ldots\rangle^-,
 \end{multline*}
the terms on the right cancel by \eqref{eqn:anti}. 
  \end{proof}
  A straightforward check shows that $\Delta$ descends to a $\bZ$-linear homomorphism
  $$
\Delta^-: \mathcal M_{n'+n''}^- (G)\ra \mathcal M_{n'}^- (G') \otimes \mathcal M_{n''}^- (G'').
$$

Denote by $\mathcal G_{\bullet}$ a flag of subgroups 
$$
0=G_{\le 0}\subsetneq G_{\le 1}\subsetneq \ldots \subsetneq G_{\le r}=G,
$$ 
let $r$ be its length.
Consider the diagram of homomorphisms
\begin{multline*}
\mathcal M_n^- (G)\rightleftarrows \bigoplus_{\substack{n_1+n_2 =n\\ \text{$\mathcal G_{\bullet}$ of lengths $2$}}}      
\mathcal M_{n_1}^- (gr_1(\mathcal G_{\bullet}))\otimes  \mathcal M_{n_1}^- (gr_2(\mathcal G_{\bullet}))
\\
 \rightleftarrows
\bigoplus_{\substack{n_1+n_2+n_3=n \\ \text{$\mathcal G_{\bullet}$ of lengths $3$}}}  
    \mathcal M_{n_1}^- 
(gr_1(\mathcal G_{\bullet})) \otimes  \mathcal M_{n_2}^- (gr_2(\mathcal G_{\bullet}))\otimes  \mathcal M_{n_3}^- (gr_3(\mathcal G_{\bullet}))
 \rightleftarrows \cdots 
\end{multline*}
where the right arrows are the natural simplicial extensions of the co-multiplication $\Delta^-$ (given by alternating sums) and the left arrows are corresponding extensions of the multiplication maps. We obtain two complexes, with differentials denoted by $d_{\Delta}$ and $d_{\nabla}$ of degree $(+1)$ and $(-1)$, respectively. 

\begin{conj}
\label{conj:1}
The cohomology of the complex with differential $d_{\Delta}$, after tensoring by $\bQ$,  is concentrated in degree 0. 
\end{conj}

This conjecture, as well as the vanishing of cohomology of the other complex, with differential $d_{\nabla}$, after tensoring with $\bQ$, would follow from:

\begin{conj}
\label{conj:int}
The operator 
$$
\bm{\Delta}:=d_{\Delta} \circ d_{\nabla} +d_{\nabla}\circ d_{\Delta},
$$
acting on each term of the complex, is invertible, after tensoring with $\bQ$, except on the first term $\mathcal M_n^- (G)$.
\end{conj}

We define
\begin{equation}
\label{eqn:prim}
\mathcal M_{n,prim}^-(G) := \mathrm{Ker} \left(\mathcal M_n^-(G)\ra  \bigoplus_{\substack{n'+n'' =n,\\ n',n''\ge 1\\  0\subsetneq  G' \subsetneq G}}
\mathcal M_{n'}^-(G') \otimes  \mathcal M_{n''}^-(G/G')\right),
\end{equation}
this is the cohomology of the complex in degree 0, with differential $d_{\Delta}$. Alternatively, assuming Conjecture~\ref{conj:1},  
we could also put
\begin{equation}
\label{eqn:prim2}
\mathcal M_{n,prim}^-(G) := \mathrm{Coker} \left(\mathcal M_n^-(G)\leftarrow \bigoplus_{\substack{n'+n'' =n, \\n',n''\ge 1\\  0\subsetneq  G' \subsetneq G}}
\mathcal M_{n'}^-(G') \otimes  \mathcal M_{n''}^-(G/G')\right),
\end{equation}
this is the cohomology of the complex in degree 0, with differential $d_{\nabla}$.

The conjecture implies that there are two {\em canonical} isomorphisms
\begin{multline*}
\label{eqn:3}
\mathcal M_{n}^-(G) \otimes \bQ  \simeq \\
\bigoplus_r \bigoplus_{\substack{n_1+\cdots + n_r =n\\ \text{$\mathcal G_{\bullet}$ of lengths $r$}}} 
\mathcal M_{n_1,prim}^- (gr_1(\mathcal G_{\bullet})) \otimes \cdots \otimes\mathcal M_{n_r,prim}^- (gr_r(\mathcal G_{\bullet}))  \otimes \bQ,
\end{multline*}
one respecting the structure of multiplication $\nabla$ and the other the structure of co-multiplication $\Delta$.

\begin{rema}
Let $\tilde{\mathcal M}_n(G)$ be the $\bZ$-module generated by $\langle a_1,\ldots ,a_n\rangle$ satisfying the symmetry condition (S),  and such that
$a_1,\ldots, a_n$ generate $A$ (as a variant one could insist that $a_j\neq 0$ for all $j$). 
The formulas for $\nabla$ and $\Delta$ are well-defined on $\tilde{\mathcal M}_n(G)$, and we denote the corresponding differentials by $d_{\tilde{\nabla}}$ and 
$d_{\tilde{\Delta}}$. 
The corresponding complex admits a surjection onto the  complex 
for $\mathcal M_n^-(G)$, with matching of both differentials. The invertibility (over $\bQ$) of 
$$
\tilde{\bm{\Delta}}:=d_{\tilde{\Delta}} \circ d_{\tilde\nabla} +d_{\tilde\nabla}\circ d_{\tilde\Delta},
$$
in some degree implies the invertibility of $\bm{\Delta}$, in the same degree. 
Note that the operator $d_{\tilde{\Delta}}$ is adjoint to $d_{\tilde{\nabla}}$, with respect to a positive-definite quadratic form, given by the identity matrix in the natural basis. Therefore, the invertibility of $\tilde{\bm{\Delta}}$ is equivalent to the vanishing of cohomology modulo torsion in this degree, 
with respect to the differential $d_{\tilde{\nabla}}$. This last question is a purely combinatorial one, concerning only abelian groups and generating properties of 
collections of elements. We expect that it will be more tractable than the original conjecture. 
\end{rema}

Consider the diagram of homomorphisms
\begin{multline*}
\mathcal M_n(G)\ra  \bigoplus_{\substack{n_1+n_2 =n\\ \text{$\mathcal G_{\bullet}$ of lengths $2$}}}      
\mathcal M_{n_1}(gr_1(\mathcal G_{\bullet})) \otimes  \mathcal M_{n_1}^- (gr_2(\mathcal G_{\bullet}))
\\
 \ra 
\bigoplus_{\substack{n_1+n_2+n_3=n \\ \text{$\mathcal G_{\bullet}$ of lengths $3$}}}  
    \mathcal M_{n_1}
(gr_1(\mathcal G_{\bullet})) \otimes  \mathcal M_{n_2}^- (gr_2(\mathcal G_{\bullet}))\otimes  \mathcal M_{n_3}^-(gr_3(\mathcal G_{\bullet}))
 \ra \cdots 
\end{multline*}
where
\begin{itemize}
\item 
 $\mathcal G_{\bullet}$ is a flag of subgroups of type
$$
0=G_{\le 0}\subseteq G_{\le 1}\subsetneq \ldots \subsetneq G_{\le r}=G, \quad r\ge 1, 
$$ 
with strict inclusions, except in the first step;
\item 
in each term, the leftmost factor is the full group, and not the quotient by the relation \eqref{eqn:anti}. 
\end{itemize}
Here the differential uses both maps $\Delta$ and 
$\Delta^-$.  Again, this is a complex; notice that here we do not have a dual differential in the other direction.

\begin{conj}
\label{conj:2}
The cohomology of the above complex, after tensoring by $\bQ$, is concentrated in degree 0. 
\end{conj}

\

We now define
\begin{equation}
\label{eqn:prim}
\mathcal M_{n,prim}(G) = \mathrm{Ker} \left(\mathcal M_n(G)\ra \bigoplus_{\substack{n'+n'' =n\\ n',n''\ge 1\\  0\subseteq  G' \subsetneq G}}
\mathcal M_{n'}(G') \otimes  \mathcal M_{n''}^-(G/G')\right),
\end{equation}
this is the cohomology of the complex in degree 0; note that the inclusion $G'$ could be trivial. 
We have
$$
\mathcal M_1(G)=\mathcal M_{1,prim}(G)
$$
for all $G$; when $G=1=\bZ/1\bZ$ we have 
$$
\mathcal M_1(1)=\bZ,  \quad \mathcal M_{n}(1) =  \mathcal M_{n,prim}(1)=0, \text{ for } n\ge 2.
$$ 
Conjecture~\ref{conj:2} implies that there is a {\em noncanonical} isomorphism
\begin{multline*}
\mathcal M_{n}(G) \otimes \bQ \simeq  \\
\bigoplus_r \bigoplus_{\substack{n_1+\cdots + n_r =n\\ \text{$\mathcal G_{\bullet}$ of lengths $r$}}} 
\mathcal M_{n_1,prim}(gr_1(\mathcal G_{\bullet}))\otimes \cdots \otimes\mathcal M_{n_r,prim}^-(gr_r(\mathcal G_{\bullet}))  \otimes \bQ.
\end{multline*}


Computer experiments (see Section~\ref{sect:exp}) suggest that, for all $N\ge 1$:
\begin{itemize}
\item 
$$
\mathcal M_{2,prim}(\bZ/N\bZ)\otimes \bQ = \mathcal M_{2,prim}^-(\bZ/N\bZ)\otimes \bQ
$$ 
and is equal to the dimension of the space of cusp forms of weight 2 for $\Gamma_1(N)$; we will discuss this in Section~\ref{sect:modular},
\item
$$
\mathcal M_{3,prim}(\bZ/N\bZ)\otimes \bQ = \mathcal M_{3,prim}^-(\bZ/N\bZ)\otimes \bQ
$$
and is equal to the number of certain cuspidal automorphic representations for a congruence subgroup of $\mathrm{GL}_3(\bZ)$, 
generated by a vector invariant under a congruence subgroup,
\item 
$$
\hskip 2.3cm 
\mathcal M_{n,prim}(\bZ/N\bZ)\otimes \bQ = \mathcal M_{n,prim}^-(\bZ/N\bZ)\otimes \bQ = 0, \quad n\ge 4,
$$
\end{itemize}
Conjectures \ref{conj:1} and \ref{conj:2} would allow us to compute $\bQ$-ranks of $\mathcal M_n(\bZ/N\bZ)$ using the
\begin{itemize}
\item 
Euler function:
\begin{align*}
\dim(\mathcal M_{1,prim} (\bZ/N\bZ)\otimes \bQ)   & = \phi(N), \quad N\ge 1 \\
\dim(\mathcal M_{1,prim}^- (\bZ/N\bZ)\otimes \bQ) & = \begin{cases} 0  & N=2, \\ \phi(N)/2 &  N\ge 3.\end{cases}
\end{align*}

\item 
well-known dimensions of the spaces of cusp-forms for $\Gamma_1(N)$, which are given by closed formulas in $N$, e.g., 

\

\centerline{\hskip 0.9cm
\small{
\begin{tabular}{c|c|c|c|c|c|c|c|c|c|c|c|c|c|c|c|c|c|c}
$N$  & ...& 11 &12 &13 &14& 15  &16 & 17 &18&19&20&\dots &180 &181\\
\hline
\hline
\rule{0pt}{3ex}
 &  0 & 1& 0& 2& 1& 1 & 2& 5 &2 & 7 & 3 &\dots & 705  & 1276 \\
\end{tabular}
}
}

\

\item somewhat
mysterious dimensions in the case $n=3$, e.g, 

\

\centerline{
\small{
\begin{tabular}{c|c|c|c|c|c|c|c|c|c|c|c|c|c|c|c|c|c|c}
$N$  & 43 &   51 &52 & 59 & 63& 67  & 68 & 72 & 73 & 75 & \dots  & 239 & 240 \\
\hline
\hline
\rule{0pt}{3ex}
         & 1  &1      & 1 & 1   & 2  &2       & 1 & 1 & 8 & 4& \dots &   3& 22 \\
\end{tabular}
}
}

\end{itemize}

\begin{exam}
\label{exam:explicit}
Our conjectures would imply that
$$
\dim( \mathcal M_n(\bZ/3^{n-1 } \bZ) \otimes \bQ) =1, \quad n\ge 1,
$$
coming from the term
$$
\mathcal M_{1,prim}(\bZ/1 \bZ) \otimes  \underbrace{ \mathcal M_{1,prim}^-(\bZ/3\bZ)\otimes \cdots \otimes \mathcal M_{1,prim}^-(\bZ/3\bZ)}_{(n-1) \text {times} }.
$$
The co-multiplication maps $\Delta, \Delta^-$ give homomorphisms 
$$
\Hom(\mathcal M_{n_1}^{(-)}(G), \bQ)\otimes \Hom(\mathcal M_{n_2}^{(-)}(G), \bQ) \ra \Hom(\mathcal M_{n}^{(-)}(G), \bQ).
$$
Using explicit nonzero elements 
$$
(\langle 0\rangle\mapsto 1 ) \in \Hom(\mathcal M_{1}(\bZ/1\bZ), \bQ), 
$$
$$    
(\langle \pm 1 \text{ mod } 3 \rangle^- \mapsto  \pm 1 )\in  \Hom(\mathcal M_{1}^-(\bZ/3\bZ), \bQ)
$$
we obtain an explicit functional on 
$\mathcal M_n(\bZ/3^{n-1 }\bZ)$ which maps 
$$
\langle 1 \text{ mod } 3^{n-1}, 3  \text{ mod } 3^{n-1}, \ldots , 3^{n-1} \text{ mod } 3^{n-1} \rangle \mapsto 1,
$$
hence is nonzero. In particular, we have 
$$
\dim(\mathcal M_{n}(\bZ/3^{n-1}\bZ)\otimes \bQ)\ge 1.
$$
Similarly, one can show that
$$
\dim(\mathcal M_{n}(\bZ/2^{n-1}\bZ)\otimes \mathbb F_2)\ge 1,
$$
Thus we obtain explicit nontrivial invariants for equivariant birational actions of $G=\bZ/3^{n-1}\bZ$ on $n$-dimensional varieties. 
Surprisingly, experiments show that 
the nontrivial invariant in 
$\Hom( \mathcal M_{n}(\bZ/2^{n-1}\bZ), \mathbb F_2)$ lifts to the trivial element in 
$\Hom( \mathcal B_{n}(\bZ/2^{n-1}\bZ), \mathbb F_2)$, for $n=2,3,4,5$.  
\end{exam}

Experiments suggest that 
$$
\dim(\mathcal M_n(\bZ/N\bZ) \otimes \bQ)=0, \quad  \text{ for all } N < 3^{n-1}
$$
and 
$$
\dim(\mathcal M_n(\bZ/N\bZ) \otimes \mathbb F_2)=0, \quad \text{ for all } N < 2^{n-1}.
$$
Moreover, 
$$
\dim(\mathcal B_n(\bZ/N\bZ) \otimes \mathbb F_2)=0,\, \quad \text{ for all }  \begin{cases} N < 2^{n}-1 & n=2,3, \\  N < 2^{n-1} & n\ge 4. \end{cases}
$$

 \section{Hecke operators}
 \label{sect:hecke}

In this section, we define analogs of Hecke operators on $\mathcal M_n(G)$. 
Fix a prime $\ell$ not dividing $\# G$ and an integer $1\le r\le n-1$. 
Put 
\begin{equation}
\label{eqn:hecke}
T_{\ell,r} (\psi(\mathbf L, \chi, \Lambda)):= 
\sum_{\mathbf L\subset \mathbf L' \subset \mathbf L\otimes \bR,  \, 
\mathbf L'/\mathbf L\simeq (\bZ/\ell\bZ)^r} \, \, \psi(\mathbf L', \chi, \Lambda),
\end{equation}
where $\chi$ is now interpreted as an element of $\mathbf L'\otimes A$, via inclusion 
$$
\mathbf L \otimes A \subset \mathbf L' \otimes  A,
$$
the surjectivity property for $\chi\in \mathbf L'\otimes A$ follows from the surjectivity of $\chi\in \mathbf L\otimes A $ and the assumption on 
coprimality of $\ell$ and the order of $G$. 

\begin{prop}
\label{prop:hecke}
The Hecke operators $T_{\ell, r}$ are well-defined on $\mathcal M_n(G)$, and commute with each other. 
\end{prop}

\begin{proof} 
Follows from the additivity of \eqref{eqn:Chi} and \eqref{eqn:hecke}.
\end{proof}

 \begin{exam}
 \label{exam:mot}
We consider the case $n=2$ and $G=\bZ/N\bZ\simeq A$. 
Then  $ \mathcal M_n(G)$ is generated by 
$$ 
\langle a_1,a_2\rangle , \quad a_1,a_2 \in \bZ/N\bZ,  \quad \gcd(a_1,a_2,N)=1,
$$ 
 such that 
 \begin{itemize}
 \item 
 $ \langle a_1,a_2\rangle = \langle a_2,a_1\rangle$,
 \item 
 $ \langle a_1,a_2\rangle = \langle a_1,a_2-a_1\rangle + \langle a_1-a_2,a_2\rangle$, for all $a_1,a_2$. 
 \end{itemize}

We write down an example of a Hecke operator on $\mathcal M_2(G)$. For each $\ell$ coprime to $N$ we have only one 
Hecke operator $T_{\ell}=T_{\ell,1}$. 

Assume that $N$ is odd and $\ell=2$. Let 
$$
\mathbf L=\bZ^2, \quad \chi= (1,0)\otimes a_1 + (0,1)\otimes a_2, \quad  \Lambda = \bR_{\ge 0}^2,
$$
the standard octant. There are three overlattices of $\mathbf L$ of index 2, corresponding to the three elements of $\bP^1(\mathbb F_2)$:
\begin{itemize}
\item $\mathbf L_0':=\bZ \cdot (\frac{1}{2}, 0) + \bZ\cdot (0,1),$
\item $\mathbf L_1':=\bZ \cdot (\frac{1}{2}, \frac{1}{2}) + \bZ\cdot (0,1) =\bZ \cdot (\frac{1}{2}, \frac{1}{2}) + \bZ\cdot (1,0),$
\item $\mathbf L_\infty':=\bZ \cdot (0,\frac{1}{2}) + \bZ\cdot (1,0)$.
\end{itemize}
The corresponding cones in the first and third case are basic simplicial, whereas in the second case it is not basic and can be 
decomposed in the union of two basic simplicial cones, with respect to $\mathbf L_1'$: 
$$
\Lambda=\Lambda_1\cup \Lambda_2,
$$
$$
\Lambda_1 =\bR_{\ge 0} \cdot (1,0) + \bR_{\ge 0} \cdot (1,1), \quad  
\Lambda_2 =\bR_{\ge 0} \cdot (1,1) + \bR_{\ge 0} \cdot (0,1).
$$
Therefore, 
\begin{equation*}
\label{eqn:hecke}
T_2 (\langle a_1,a_2\rangle ) = \langle 2a_1,a_2\rangle +  
\big(\langle a_1-a_2,2a_2\rangle + \langle 2a_1,a_2-a_1\rangle\big)
 + \langle a_1,2a_2\rangle.
\end{equation*}
The middle term follows from equalities
$$
e_1\otimes a_1 + e_2\otimes a_2 \!=\! e_1\otimes (a_1-a_2) + \frac{e_1+e_2}{2} \otimes 2a_2 \!= \!
\frac{e_1+e_2}{2} \otimes 2a_1 + e_2\otimes (a_2-a_1).
$$
We leave it as an exercise to write down a similar formula for the action of  
$T_3$ on $\mathcal M_2(G)$ and $T_2$ on $\mathcal M_3(G)$. 
 \end{exam}

 To define Hecke operators  $T^*_{\ell,r}$ in the co-vector version, we consider 
 sublattices $\mathbf L'\subset \mathbf L$, of index $\ell^r$, such that the quotient is isomorphic to 
$(\bZ/\ell \bZ)^r$. 
In particular, $T^*_2=T^*_{2,1}$ on $\mathcal M_2^*(G)$ is given by 
$$
T_2^*( [a_1,a_2]^*) = [2a_1,a_2]^* + [2a_1,a_1+a_2]^* + [a_1+a_2, 2a_2]^*+ 
[a_1,2a_2]^*
$$
and $T_{2,1}^*$ on $\mathcal M_3(G)$ by 
$$
T_{2,1}^*( [a_1,a_2,a_3]^*) = 
[2a_1,a_2,a_3]^* + [a_1,2a_2,a_3]^* + [a_1,a_2, 2a_3]^*+
$$
$$
+ [2a_1,a_1+a_2,a_3]^* + [a_1+a_2,2a_2,a_3]^* +
 $$
 $$
 + [a_1,2a_2, a_2+a_3]^*+ [a_1,a_2+a_3,2a_3]^* +
$$
$$
+ [2a_1,a_2,a_1+a_3]^* + [a_1+a_3,a_2,a_3]^*  + 
 $$
 $$
 + [2a_1,a_1+a_2, a_1+a_3]^*
 + [a_1+a_2,2a_2, a_2+a_3]^*  + [a_1+a_3,a_2+a_3, 2a_3]^*+ 
 $$
 $$
+ [a_1+a_2,a_2+a_3, a_1+a_3]^*.
 $$ 
  
 \begin{rema}
The groups $\mathcal M_{n,prim}(G)^-$ defined in \eqref{eqn:prim} are preserved under the action of Hecke operators. 
 \end{rema}

\

 \centerline{\bf Part 2}
 
 \section{Refined birational invariants} 
 \label{sect:bira}
 
 There is a refinement of $\mathcal B_n(G)$, connecting it to the Burnside group of varieties considered in \cite{KT}. 
 Let $K$ be an algebraically closed field of characteristic zero. 
 Let
 $$
 \Bir_{n-1,m}(K), \quad 0\le m\le n-1,
 $$
 be the set of equivalence classes of $(n-1)$-dimensional irreducible varieties over $K$, 
 modulo $K$-birational equivalence,
 which are $K$-birational 
 to products $W\times \mathbb A^{m}$, and not to $W'\times \mathbb A^{m+1}$, for any $W'$. 
Let
 $$
 \mathcal B_n(G,K) := \oplus_{m=0}^{n-1} \oplus _{[Y] \in \Bir_{n-1,m}(K)} \mathcal B_{m+1}(G), 
 $$
 with
 $$
 \mathcal B_1(G) = \begin{cases} 
 \oplus_{a\in (\bZ/N\bZ)^\times}\,\, \bZ & \text{ if } G=\bZ/N\bZ, \, N\ge 2, \\
0 							  & \text{ if $G$ is not cyclic}.
\end{cases}
 $$
Let $X$ be an irreducible $K$-variety with a generically free action of $G$. As in Section~\ref{sect:intro}, we may assume
 that $G$ acts regularly; let 
 $$
 X^G = \sqcup_{\alpha} F_{\alpha}
 $$
 be the decomposition of the fixed point locus into irreducible, disjoint, components.  
 The spectrum for the $G$-action in the tangent space to $X$ at any point $x_{\alpha} \in F_{\alpha}$ is given by
 $$
 a_1,\ldots, a_{n-\dim(F_{\alpha})}, \underbrace{0, \ldots }_{\dim(F_{\alpha})}, \quad a_i\neq 0.
 $$
 
 Define
 $$
 \beta_K(X) \in  \mathcal B_n(G,K)
 $$
by taking into account the birational types of fixed loci under $G$, as follows:
write
$$
Y_{\alpha}: =  F_{\alpha} \times \bA^{n-1-\dim(F_{\alpha})}
$$
and let $m_{\alpha}\in \bZ_{>0}$ be the maximal integer such that 
$$
Y_{\alpha}\sim Z_{\alpha}\times \mathbb A^{m_{\alpha}}, 
$$ 
 clearly, 
 $$
 m_{\alpha}\ge n-1-\dim(F_{\alpha}).
 $$ 
 Then 
 $$
 \beta_K(X) = \sum_{\alpha}\beta_{\alpha}(X),
 $$
where
$$
\beta_{\alpha}(X) = [a_1,\ldots, a_{n-\dim(F_{\alpha})},  
\underbrace{0, \ldots }_{n-1-m_{\alpha}}] \in \text{ copy of } \mathcal B_{m_{\alpha}+1}(G), 
$$
labeled by the birational type of $Y_{\alpha}$. 

The invariance under blowups follows from the fact that all $(n-1)$-dimensional birational types   arising as labels in each particular subcase of  the proof of Theorem~\ref{thm:main}
 coincide with each other.
 
 \section{Hecke operators: variants}
 \label{sect:variants}
 
Let $G$ be a finite abelian group and $A$ its group of characters. 
Another variant concerns coefficients. It works both for the vector and co-vector versions.
For simplicity, we consider symbols with coefficients in $\bQ$. 
Consider an irreducible algebraic representation 
$$
\rho_{\lambda} : \mathrm{GL}_n(\bQ)\ra \Aut(\mathsf V_{\lambda}),
$$
with highest weight 
$$
\lambda = (\lambda_1\le  \ldots \le  \lambda_n), \quad \lambda_i\in \bZ.
$$

The representation $\rho_{\lambda}$ 
defines
 a functor from the groupoid of $n$-dimensional $\bQ$-vector spaces to the 
category $\mathrm{Vect}_{\bQ}$ of all $\bQ$-vector spaces, which we denote by the same letter. 
In particular, for any lattice $\mathbf L$ of rank $n$ we can speak of
$$
\rho_{\lambda}(\mathbf L\otimes \bQ) \in \mathrm{Vect}_{\bQ}.
$$
For example, if $\rho_{\lambda}$ is the $m$-th symmetric power $\Sym^m(V)$ of the standard representation, i.e., 
$\lambda= (0,\ldots, 0, m)$, then  
$$
\rho_{\lambda} (\mathbf L\otimes \bQ) = \Sym^m(\mathbf L\otimes \bQ).
$$

Consider the $\bQ$-vector space
$$
\mathcal M_n(G, \rho_{\lambda})
$$
generated by 
symbols
$$
\psi(\mathbf L, \chi, \Lambda, v),
$$
on isomorphism classes of quadrupels, where
$\mathbf L, \chi, \Lambda$ are as in Section~\ref{sect:hecke} and 
$$
v\in  \rho_{\lambda} (\mathbf L\otimes \bQ),
$$ 
subject to relations
\begin{itemize}
\item 
$\psi(\mathbf L, \chi, \Lambda, v_1+v_2)= \psi(\mathbf L, \chi, \Lambda, v_1) + \psi(\mathbf L, \chi, \Lambda, v_2)$,

\

\item 
$
\psi(\mathbf L, \chi, \Lambda, v) = \sum_{i=1}^k \psi(\mathbf L, \chi, \Lambda_i, v),
$ for any  decomposition 
$$
\Lambda= \Lambda_1\cup \cdots \cup \Lambda_k.
$$
\end{itemize}
Here, one can assume that subcones $\Lambda_i$ are basic simplicial and that the decomposition is standard, as in 
Section~\ref{sect:hecke}, or simply that $\Lambda_i$ are finitely-generated rational subcones of full dimension, with 
disjoint interiors. 
The action of Hecke operators on  
$\mathcal M_n(G, \rho_{\lambda})$ is defined as in \eqref{eqn:hecke}. 

The co-vector version of this construction is straightforward. 

\begin{rema}
\label{rema:hope}
We expect that for $n=2$, $G=\bZ/N\bZ$, and $\rho_{\lambda}$ given by the $m$-th symmetric power, 
the $\bQ$-vector spaces $\mathcal M_n(G, \rho_{\lambda})$, endowed with the action of Hecke operators $T_{\ell,r}$,
are related to modular forms of weight $(m+2)$, for the congruence subgroup $\Gamma_1(N)$. 
\end{rema}

\section{Algebraic versions of automorphic forms}
\label{sect:alge}

A further generalization of results in Section~\ref{sect:variants} takes place in the following context. 
Let $\mathsf G$ be a connected reductive group over $\bQ$. There is a notion of admissible 
Harish-Chandra modules  $\mathcal E$  for $\mathsf G(\bR)$: these are $\bC$-vector spaces of countable dimension, endowed with an action of 
the maximal compact subgroup $\mathsf  K\subset \mathsf G(\bR)$ and a compatible 
action of the complexified Lie algebra $\mathfrak g_{\bC}=\mathrm{Lie}(G)\otimes \bC$. The action of $\mathsf K$ 
decomposes $\mathcal E$ as a countable sum of finite-dimensional representations of 
$\mathsf K$, each appearing with finite multiplicity. 
We assume that the center $\mathfrak  z \subset \mathfrak U(\mathfrak g)$ 
acts by scalars, called the central character of $\mathcal E$.  
The group $\mathsf G(\bR)$ acts on the Schwartz completion of $\mathcal S(\mathcal E)$.
Let $\mathcal S(\mathcal E)'$ be the {\em continuous} dual space, it is a subspace of the {\em algebraic} dual space $\mathcal E^\vee$. 
The
 congruence subgroups of $\mathsf G(\bQ)$ have finite-dimensional invariants in   $\mathcal S(\mathcal E)'$. One can view the theory of automorphic forms as the study 
of these finite-dimensional spaces of invariants, together with the action of a Hecke algebra. Note that in the last
step we consider $\mathcal S(\mathcal E)'$ only as a $\mathsf G(\bQ)$-module, and not as a $\mathsf G(\bR)$-module.  

Almost all automorphic forms are not related to motives or Galois representations; the part relevant for number theory (called {\em algebraic} automorphic  forms) is specified by a certain integrality constraint on the central character.

Returning to considerations above, we see that we can imitate the theory of automorphic forms, with representations of $\mathsf G(\bQ)$ in $\mathcal S(\mathcal E)'$, by a {\em different} class of representations 
of $\mathsf G(\bQ)$, defined over $\bQ$. Assume that $\mathsf G=\mathrm{GL}_n$, over $\bQ$.
Let 
\begin{equation}
\label{eqn:f}
\mathcal F_n = \left<  \mathcal X_{\Lambda}  \right>_{\otimes \bQ} = \mathcal F_{\mathbf L, \bZ}\otimes \bQ,\quad \text{ for } \mathbf L=\bZ^n, 
\end{equation}
be 
the $\bQ$-vector space generated by 
characteristic functions $\mathcal X_{\Lambda}$ of 
convex finitely generated rational polyhedral cones $\Lambda\subset  \bR^n$, modulo 
functions with support of dimension $\le (n-1)$. 
Note that 
$$
\mathcal F_n  \subset \mathrm L_{\infty}(\bR^n),
$$
the space of bounded measurable functions. Clearly, $\mathsf G(\bQ) = \mathrm{GL}_n(\bQ)$ acts on 
$\mathcal F_n$. 
Let 
$$
\rho=\rho_{\lambda} : \mathrm{GL}_n(\bQ)\ra \Aut(\mathsf V_{\lambda}) 
$$
be a finite-dimensional irreducible representation as above.
Let 
$$
\Gamma\subset \mathrm{GL}_n(\bQ)
$$ 
be an arithmetic subgroup. 
The spaces of invariants, respectively, coinvariants
\begin{equation}
\label{eqn:heckeco}
H^0(\Gamma, \mathcal F_n^{\vee} \otimes \mathsf V_{\lambda}^\vee) =
(\mathcal F_n^{\vee} \otimes \mathsf V_{\lambda}^\vee) ^{\Gamma}, \quad \quad 
H_0(\Gamma, \mathcal F_n \otimes \mathsf V_{\lambda}) =
(\mathcal F_n \otimes \mathsf V_{\lambda})_{\Gamma},
\end{equation}
are dual to each other finite-dimensional spaces, since the module of characteristic functions is finitely-generated over the group ring
of the arithmetic subgroup $\Gamma$. 

For example, for $n\ge 2$, if $\rho$ is the trivial representation, and $\Gamma\subset \mathrm{GL}_n(\bZ)=\mathrm{Aut}(\mathbf L)$ is the stabilizer of the vector 
$$\chi=(1,0,0,\ldots)\in \mathbf L \otimes \bZ/N\bZ$$
then the group of coinvariants is  (up to torsion) our group $\mathcal M_n(\bZ/N\bZ)$. Similarly, by taking the stabilizer of the coordinate co-vector modulo $N$, we obtain the co-vector version
 $\mathcal M_n^*(\bZ/N\bZ)$.

More generally, for any finite abelian group $G$  with character group $A$ such that $G$ can be generated by at least $n$ elements let us choose an
element 
$$
\chi  \in \mathbf{L} \otimes A, \quad \mathbf{L} =  \bZ^n,
$$
such that the induced homomorphism
$
\mathbf{L}^\vee\ra A 
$
is surjective. 
We define
$$
\Gamma(G,n)\subset \mathrm{GL}_n(\bZ)
$$
as the stabilizer of $\chi$. Note that the conjugacy class of the stabilizer does not depend on the choice of $\chi$. 
Then, for $n\ge 2$ and such that $G$ is generated by at most $n$ elements, we have 
\begin{equation}
\label{eqn:fn}
\mathcal M_n(G)\otimes \bQ = H_0(\Gamma(G,n), \mathcal F_n).
\end{equation}

A key observation is that $\mathcal F_n$ is a $\mathrm{GL}_n(\bQ)$-module
which is {\em finitely generated} as $\mathrm{GL}_n(\bZ)$-module; moreover, 
\begin{equation}
\label{eqn:res}
\mathrm{Res}^{\mathrm{GL}_n(\bQ)}_{ \mathrm{GL}_n(\bZ)} (\mathcal F_n)\in 
\mathrm{Perf}(\bQ[\mathrm{GL}_n(\bZ)]-\mathrm{mod}),
\end{equation}
i.e., 
$\mathcal F_n$, considered as a  $\mathrm{GL}_n(\bZ)$-module, 
admits a finite-length resolution by finitely-generated projective 
modules over the group ring of $\mathrm{GL}_n(\bZ)$ (see Proposition~\ref{prop:stein}).

\begin{ques}
\label{ques:maybe}
Are there other interesting $\mathrm{GL}_n(\bQ)$-modules which are finitely-generated as $\mathrm{GL}_n(\bZ)$-modules, or more
strongly, belong to  
$$
\mathrm{Perf}(\bQ[\mathrm{GL}_n(\bZ)]-\mathrm{mod})?
$$
\end{ques}

An even more general question would be to find a bounded from above complex of representations of $\mathsf G(\bQ)$ which, after restriction to $\mathsf G(\bZ)$, is quasi-isomorphic to a complex of finitely-generated projective modules over the group ring.

Both $\bQ$-vector spaces in \eqref{eqn:heckeco} carry actions of Hecke operators, which have algebraic eigenvalues in these spaces. 
By \eqref{eqn:res}, 
$$
\dim(H_i(\Gamma, \mathcal F_n\otimes \mathsf V_{\lambda}))< \infty, \quad \text{ for all } i\ge 0, 
$$
and the spaces, for $i\ge 1$, also carry actions of Hecke operators with algebraic eigenvalues.

We will see below that our representation $\mathcal F_n$ falls into a well-studied subclass of {\em cohomological} automorphic forms, i.e., those realized in cohomology of 
arithmetic groups with coefficients in finite-dimensional representations $\rho$.

Recall the definition of {\em Steinberg} modules: Let $V/\bQ$ be a $\bQ$-vector space of dimension $n\ge 0$, and $\mathcal T_n$ 
the simplicial complex of flags of $\bQ$-vector subspaces of $V$, i.e., the geometric realization of the poset of nontrivial subspaces in $V$. 
Put
$$
\mathrm{St}(V):= \begin{cases}  H_{n-2}(\mathcal T_n,\bZ) & n\ge 3 \\
                                 \text{$\bZ$-combinations of lines in $V$ with total weight 0}             & n=2 \\
                                 	\bZ & n=0, 1.
\end{cases}
$$  

This is a representation of $\mathrm{Aut}(V)$, which we denote by $\mathrm{St}_n$ for $V=\bQ^n$.
One of the roles of the Steinberg module is as a dualizing module, in  the sense that 
$$
H_i(\mathrm{SL}_n(\bZ),\mathrm{St}_n\otimes M)=H^{n(n-1)/2-i}(\mathrm{SL}_n(\bZ),M),
$$
for any representation $M$ of $\mathrm{SL}_n(\bZ)$ with coefficients in $\bQ$.

Let 
$$
\mathcal F(V) = \mathcal F_n,
 $$
 as  in \eqref{eqn:f}, 
 where the identification depends on the choice of a basis of $V$, different choices are related by the action of $\mathrm{G}_n(\bQ)$ on $\mathcal F_n$. 
It has a filtration by submodules
$$
0\subset \mathcal F^{\le 0}(V)\subset \mathcal F^{\le 1}(V)\subset \cdots \subset \mathcal F^{\le n}(V)=\mathcal F(V),
$$
where $\mathcal F^{\le i}(V)$ are generated by functions pulled back from quotient spaces of dimension $i$. In particular, 
$$
\mathcal F^{\le 0}(V)=\bZ = \{ \text{constant $\bZ$-valued functions on $V$} \} .
$$

The following fact is presumably well-known: 

\begin{prop}
\label{prop:stein}
$$
\mathrm{gr}^i(\mathcal F(V))=\oplus_{V\twoheadrightarrow V', \dim(V')=i} \,\,\,  \mathrm{St}(V')\otimes or(V'),
$$
where $or(V')$ is the 1-dimensional  $\bZ$-module of orientations of $V'$, i.e., $\mathrm{GL}(V')$ acts via the sign 
of the determinant.
\end{prop}

\begin{proof}
Let us first prove that 
$$
\mathrm{gr}^n(\mathcal F(V))=\mathcal F(V)/\mathcal F^{\le n-1}(V)
$$ 
is isomorphic to 
$$
 \mathrm{St}(V)\otimes or(V).
 $$ 
We apply the Fourier transform to elements of $\mathcal F(V)$ viewed as distributions with moderate growth on $V\otimes\bR\simeq \bR^n$. 

For example, the Fourier transform of the characteristic function of the standard coordinate octant $(\bR_{\ge 0})^n$
is equal to the distribution 
$$
\prod_{i=1}^n(\sqrt{-1} \, v.p.(1/x_i)+\pi\delta(x_i))\prod_{i=1}^n|dx_i|
$$ 
with values in volume forms, 
where $v.p.(1/x)$ is the unique odd distribution of homogeneity degree $-1$ on $\bR^1$ equal to $1/x$ on $\bR\setminus 0$.

The image of $\mathcal F^{\le n-1}(V)$ is characterized by the property that the support of the distribution is contained in a finite union of hyperplanes.
Therefore,  the quotient group $\mathcal F(V)/\mathcal F^{\le n-1}(V)$ is  identified with the abelian group generated by volume elements on the dual space $(V\otimes\bR)^\vee$, 
of the form 
$$(\sqrt{-1})^n
|dx_1\wedge\dots\wedge dx_n|/(x_1\cdots x_n),
$$ 
where $x_1,\dots,x_n$ are coordinates in $(V\otimes\bR)^\vee$ in a rational basis. Choosing an orientation of $V$ 
(or, equivalently, of $V^\vee$) and dividing by $(\sqrt{-1})^n$, 
we identify the latter space with top-degree meromorphic  differential forms on the vector space $V^\vee$ considered as an  \emph{algebraic variety} $\bA^n_\bQ$ over $\bQ$ spanned by forms of type $\wedge_{i=1}^n (dx_i/x_i)$ for coordinates in a rational basis. This is an alternative description of the Steinberg module. The case of deeper terms of the dimension filtration is similar.
\end{proof}

This implies that the computation of cohomology with coefficients in $\mathcal F(V)$, tensored with finite-dimensional modules, 
and, in particular, of coinvariants, would reduce to the computation of cohomology for $\mathrm{St}$-modules and their pullbacks from parabolic subgroups.
There is extensive literature on the cohomology of $\mathrm{St}$-modules (see, e.g.,  \cite{ash} and the references therein), but these computations
do not capture the potentially interesting extension data in $\mathcal F(V)$. 

\

To summarize, we have a surjective homomorphism
$$
\mathcal F_n \twoheadrightarrow \mathrm{St}_n\otimes \ori_n, 
$$
where
$$
\ori_n : \mathrm{GL}_n(\bQ)\ra \bQ^\times, \quad \gamma \mapsto  \mathrm{sgn}(\det(\gamma)).
$$
It gives rise to a surjective homomorphism
$$
H_0(\Gamma(G,n), \mathcal F_n) \twoheadrightarrow H_0(\Gamma(G,n), \mathrm{St}_n\otimes \ori_n).
$$

\begin{prop}
There exists a commutative diagram

\

\centerline{
\xymatrix{
H_0(\Gamma(G,n), \mathcal F_n) \ar[d]_{\simeq} \ar@{>>}[r] & H_0(\Gamma(G,n), \mathrm{St}_n\otimes \ori_n) \ar[d]^{\simeq} \\
\mathcal M_n(G)\otimes \bQ  \ar@{>>}[r]^{\mu^-} & \mathcal M_n^-(G)\otimes \bQ,
}
}

\

\noindent
where the horizontal arrows are the natural surjections, the left vertical arrow is the isomorphism  \eqref{eqn:fn} and the right vertical arrow is an isomorphism as well. 
\end{prop}

\begin{proof}
The proof of the commutativity of the diagram is straightforward, we explain only the right vertical isomorphism. 
Recall that the Steinberg representation $\mathrm{St}_n$ restricted to $\mathrm{GL}_n(\bZ)$ is generated by the set of $\bZ$-bases 
$$
\{ (e_1,\ldots, e_n)\},
$$
modulo relations 
\begin{itemize}
\item $(e_{\sigma(1)}, \ldots, e_{\sigma(n)}) = (-1)^n (e_1,\ldots, e_n)$, $\sigma\in \mathfrak S_n$,
\item $(e_1,e_2, e_3\ldots, e_n) = (e_1+e_2, e_2, e_3,\ldots, e_n) + (e_1,e_1+e_2,\ldots, e_n) $,
\item $(e_1,\ldots, e_n) = (-e_1,e_2, \ldots, e_n)$,
\end{itemize}
see, e.g., \cite[Theorem B]{church} and the references therein. 
Therefore, $\mathrm{St}_n\otimes \ori_n$, restricted to $\mathrm{GL}_n(\bZ)$ is again generated by 
the set of $\bZ$-bases 
$$
\{ (e_1,\ldots, e_n)\}, 
$$
but subject to new relations 
\begin{itemize}
\item $(e_{\sigma(1)}, \ldots, e_{\sigma(n)}) = (e_1,\ldots, e_n)$, $\sigma\in \mathfrak S_n$,
\item $(e_1,e_2, e_3\ldots, e_n) = (e_1+e_2, e_2, e_3,\ldots, e_n) + (e_1,e_1+e_2,\ldots, e_n) $,
\item $(e_1,\ldots, e_n) = -(-e_1,e_2, \ldots, e_n)$.
\end{itemize}
We see that the first relation is the symmetry relation (S), and the last relation the anti-symmetry relation \eqref{eqn:anti}; 
the second relation translates to relation (M) for $k=2$. 

\end{proof}

Put
$$
\mathbb H_n:= \mathrm{GL}_n(\bR) / \bR^\times_{>0} \cdot \mathrm O_n(\bR);
$$
we have, for $n\ge 2$, and $G$ generated by at most $n$ elements, 
\begin{align*}
\mathcal M_n^-(G)\otimes \bQ & = H_0(\Gamma(G,n), \mathrm{St}_n\otimes \ori_n) \\ 
& = H_{n-1}^{BM}(\Gamma(G,n) \backslash \mathbb H_n,\ori_n) \\
& = H^{\frac{n(n-1)}{2}} (\Gamma(G,n) \backslash \mathbb H_n, \ori_n^{\otimes n}) \\
& = H^{\frac{n(n-1)}{2}}  (\Gamma(G,n),\ori_n^{\otimes n}).
\end{align*}
Indeed, the generator $(e_1,\ldots, e_n) $ of  $\mathrm{St}_n$, where $e_1,\ldots, e_n$ is  the standard basis of $\bZ^n$, 
maps to the homology class of the Borel-Moore chain 
$$
(\bR^\times_{>0})^{n-1} \simeq \mathrm{Diag}_{>0,n}(\bR) / \bR^\times_{>0} \subset \mathbb H_n.
$$
The third isomorphism is Poincare duality.

Let $\Gamma\subset \mathrm{GL}_n(\bZ)$ be an arithmetic group. 
The cuspidal part of cohomology, with coefficients in a finite-dimensional representation $\rho$ of $\mathrm{GL}_n(\bQ)$, is a subspace of interior cohomology
$$
H^*_{int}(\Gamma, \rho) := \text{Image}\left(H^*_c(\Gamma \backslash \mathbb H_n, \rho) \ra H^*(\Gamma\backslash \mathbb H_n , \rho)\right).
$$
It is believed that for arithmetic subgroups of $\mathrm{GL}_n(\bQ)$, $n\le 3$, the cuspidal part {\em coincides} with the interior part. 

Notice that $\ori_n$, restricted to $\mathrm{GL}_n(\bZ)$, coincides with the algebraic representation $\det_n: \gamma\mapsto \det(\gamma)$.

It is known that 
$$
H^i_{cusp}(\Gamma, \rho)\neq 0 
$$
only for 
$$
\frac{\frac{n(n+1)}{2}-1}{2} -\frac{ [\frac{n-1}{2} ]}{2} \le     i \le \frac{\frac{n(n+1)}{2}-1}{2} + \frac{[\frac{n-1}{2}]}{2}.
$$
The upper bound coincides with $\frac{n(n-1)}{2}$ for $n=1,2,3$ and is strictly smaller for $n\ge 4$. 
Our experiments (see Section~\ref{sect:exp}) suggest that 
$$
\mathcal M_{n,prim}^-(G) = H_{cusp}^{\frac{n(n-1)}{2}}(\Gamma(G,n), \ori_n^{\otimes n}),
$$
hence vanish for $n\ge 4$.

In the following section, we will see that, for $n=2$, the main actors
are modular forms of weight 2, and sums of two Tate motives twisted by characters.

\

Among other variants in the definition of $\mathcal F$ are:
\begin{itemize}
\item using $\bZ$ or finite fields
as coefficients, instead of $\bQ$-coefficients, one can study torsion effects.  
\item one can omit the condition of factoring by characteristic functions with support in dimension $\le (n-1)$.
\item when the representation $\rho$ is on the space of degree-$d$ polynomials, one can consider
{\em polynomial splines}, with respect to some complete rational fan $\Sigma$ on $\bR^n$, i.e., 
functions on $\bR^n$ which are piecewise polynomial on the cones of $\Sigma$, with $\bQ$-coefficients, 
and with continuous derivatives up to some fixed $d'<d$. 
\end{itemize}

The last example is especially interesting as such representations are realized as submodules of extensions of Steinberg modules, and coinvariants with values in such modules could, potentially, 
capture higher homology groups of Steinberg modules, thus making them computationally much more accessible.

We finish this section with a challenge concerning the possibility, in the framework of Question~\ref{ques:maybe}, to go beyond the realm of cohomological
 (but still algebraic) automorphic forms. 
\begin{ques}
Can one find a representation of $\mathrm{SL}_2(\bQ)$ whose restriction to $\mathrm{SL}_2(\bZ)$ is finitely-generated, and whose Hecke spectrum captures modular forms of weight $1$ and Maass forms with Laplace eigenvalue $1/4$?
\end{ques}

Morally, such modules should be realized in a class of odd/even distributions on $\bR^2$ of homogeneity degree $-1$.

\section{Lattice-theoretic approach to multiplication and co-multiplication}

In this section, we reinterpret the multiplication and co-multiplication on $\mathcal M_n^-(G)$, defined in Section~\ref{sect:co}, in 
terms of lattices. 

For any $n\ge 1$ and any nontrivial finite abelian group $G$ we
define
$$
\mathcal E_n(G):=\bQ^{ \{ \text{epi } \bZ^n \twoheadrightarrow G\}},
$$
it is a finite-dimensional permutation module for $\mathrm{GL}_n(\bZ)$. 
Define the stack (with finite stabilizers)
$$
\mathbb X_n:=\mathrm{GL}_n(\bZ) \backslash  \mathrm{GL}_n(\bR) /\mathrm O_n(\bR).
$$
This stack parametrizes rank $n$ Arakelov bundles on $\widehat{\Spec(\bZ)}$, i.e., pairs $(\mathbf L, h)$, where $\mathbf L$ is a lattice  of rank $n$ 
and $h$ is a positive-definite quadratic form on $\mathbf L\otimes \bR$. Let 
$\mathcal L_{n,G}$ be a  $\bQ$-local system on $\mathbb X_n$ associated with the representation $\mathcal E_n(G)\otimes \ori_n$.
Then we have 
\begin{equation}
\label{eqn:fnnn}
\mathcal M_n^-(G)\otimes \bQ= H_n^{BM}(\mathbb X_n, \mathcal L_{n,G}).
\end{equation}

The multiplication $\nabla^-$, defined in in Section~\ref{sect:co}, admits the following reformulation in this language: 
consider flags $\mathcal G_\bullet$ of subgroups
$$
0=G_{\le 0}\subsetneq G_{\le 1}\subset \ldots \subsetneq G_{\le r} = G, \quad  r\ge 1,
$$
and sequences of positive integers 
$n_1,\ldots, n_r$, such that $n_1+\cdots + n_r=n$.
We have a homomorphism
\begin{equation}
\label{eqn:del}
\bigotimes_{i=1}^r H_{n_i}^{BM} (\mathbb X_{n_i}, \mathcal L_{n_i,gr_i(\mathcal G_\bullet)}) 
 \ra 
H_{n}^{BM} (\mathbb X_{n}, \mathcal L_{n,G}),
\end{equation}
defined as follows: 
consider the graph
$$
\mathbb Y^{\nabla}_{n_1,\ldots, n_r}\subset \left( \mathbb X_ {n_1}\times \cdots \times  \mathbb X_{n_r}\right)   \times \mathbb X_n,
$$
of the closed embedding (hence proper map)
$$
\mathbb X_ {n_1}\times \cdots\times  \mathbb X_{n_r}   \ra \mathbb X_n,
$$
given by
$$
(\mathbf L_{1}, h_1), \ldots ,(\mathbf L_{r}, h_r)  \mapsto   (\mathbf L = \mathbf L_{1}\oplus \cdots  \oplus \mathbf L_{r}, h= 
 h_{1}\boxplus \cdots  \boxplus h_r).
$$
We have a diagram

\

\centerline{
\xymatrix{
& \mathbb Y^{\nabla}_{n_1,\ldots, n} \ar[dl]_{\pi_{n_1,\ldots, n_r}} \ar[dr]^{\pi_n}  & \\
 \mathbb X_{n_1} \times \cdots \times \mathbb X_{n_r}& & \mathbb X_n 
}
}

\

\noindent
Here, $\pi_{n_1,\ldots, n_r}$ is an isomorphism. The morphism of local systems
$$
 \pi^*_{n_1,\ldots, n_r} (\mathcal L_{n_1,gr_1(\mathcal G_\bullet)} \boxtimes \cdots \boxtimes \mathcal L_{n_r,gr_r(\mathcal G_\bullet)} )\ra 
  \pi^*_{n} \mathcal L_{n,G}
$$
is given, at any point, by
\begin{itemize}
\item 
a canonical identification of orientation bundles
$$
\ori(\mathbf L_1) \otimes \cdots \otimes \ori(\mathbf L_r) \stackrel{\sim}{\lra} \ori(\mathbf L) 
$$
\item a morphism of fibers of 
local systems associated to the permutation modules 
\begin{equation}
\label{eqn:chij}
\bQ^{ \{ \text{epi } {\mathbf L}_1^\vee \twoheadrightarrow A_1\}} \otimes \cdots \otimes \bQ^{ \{ \text{epi } {\mathbf L}_r^\vee \twoheadrightarrow A_r\}}  \ra 
\bQ^{ \{ \text{epi } {\mathbf L}^\vee \twoheadrightarrow A\}};
\end{equation}
consider 
$$
\chi \in \mathbf L\otimes A := \Hom(\mathbf L^\vee,A),
$$
such that  the restriction of $\chi$ to $\mathbf L_i^\vee\subset \mathbf L^\vee$ takes values in characters of $G$ vanishing on $G_{\le i-1}$, for all $i$; 
such characters induce characters of $gr_i(\mathcal G_\bullet)$, and homomorphisms 
$$
\chi_i :\mathbf L_i^\vee \ra A_i:=\Hom(gr_i(\mathcal G_\bullet), \bC^\times),
$$
we insist that $\chi_i$ are surjective, for all $i$ (this implies that $\chi$ is surjective as well).
Such $\chi$ defines a morphism of permutation modules of rank 1, given by an elementary matrix, with indices
$$
(\chi_1,\ldots, \chi_r),\chi
$$
taking 
the sum over all such elementary matrices  
defines the 
desired homomorphism \eqref{eqn:chij}. 
\end{itemize}

The co-multiplication $\Delta^-$, defined in Section~\ref{sect:co}, also  admits a geometric reformulation. We 
 have a homomorphism
\begin{equation}
\label{eqn:del2}
H_{n}^{BM} (\mathbb X_{n}, \mathcal L_{n,G}) \ra \bigotimes_{i=1}^r H_{n_i}^{BM} (\mathbb X_{n_i}, \mathcal L_{n_i,gr_i(\mathcal G_\bullet)}) 
\end{equation}
defined similarly to \eqref{eqn:del}, but instead of the graph 
$\mathbb Y^{\nabla}_{n_1,\ldots n_r}$ of a map, 
we consider the {\em correspondence} 
$$
\mathbb Y^{\Delta}_{n_1,\ldots, n_r} \subset \mathbb X_n \times \left( \mathbb X_{n_1}\times \cdots \times \mathbb X_{n_r}\right),
$$
which is \'etale over $\mathbb X_n$ and {\em proper} over  $\left( \mathbb X_{n_1}\times \cdots \times \mathbb X_{n_r}\right)$, and which can be viewed as a graph 
of a multi-valued map. In detail, an element of $\mathbb Y_{n_1,\ldots, n_r}$ is given by the data:
\begin{itemize}
\item $(\mathbf L,h)$, a lattice of rank $n$, with a metric, i.e., a positive quadratic form $h$ on $\mathbf L\otimes \bR$ as above, 
\item flag $\mathbf L_\bullet$ of full sublattices 
$$
0=\mathbf L_{\le 0}\subsetneq \mathbf L_{\le 1} \subsetneq \ldots \subsetneq \mathbf L_{\le r}= \mathbf L.
$$
\item choice of isomorphisms 
$$
\mathbf L_i \simeq gr_i(\mathbf L_\bullet)
$$
such that the induced metrics on $\mathbf L_{n_i}\otimes \bR$ coincide with $h_i$. 
\end{itemize}
We have a diagram

\centerline{
\xymatrix{
& \mathbb Y^{\Delta}_{n_1,\ldots, n} \ar[dl]_{\pi_n} \ar[dr]^{\pi_{n_1,\ldots, n_r}} & \\
\mathbb X_n & & \mathbb X_{n_1} \times \cdots \times \mathbb X_{n_r} .
}
}

The morphism of local systems 
$$
\pi_n^* \mathcal L_{n,G} \ra \pi^*_{n_1,\ldots, n_r} 
\left(  \mathcal L_{n_1,gr_1(\mathcal G_\bullet) }  \oplus \cdots \mathcal L_{n_r,gr_r(\mathcal G_\bullet)}  \right)
$$
is given, at any point, by 
\begin{itemize}
\item a natural isomorphism of orientation bundles
$$
\ori(\mathbf L) \simeq \ori(\mathbf L_1) \otimes \cdots \otimes \ori(\mathbf L_r),
$$
\item a morphism of fibers of 
local systems associated to the permutation modules 
\begin{equation}
\label{eqn:chij2}
\bQ^{ \{ \text{epi } {\mathbf L}^\vee \twoheadrightarrow A\}} \ra 
\bQ^{ \{ \text{epi } {\mathbf L}_1^\vee \twoheadrightarrow A_1\}} \otimes \cdots \otimes \bQ^{ \{ \text{epi } {\mathbf L}_r^\vee \twoheadrightarrow A_r\}}; 
\end{equation}
consider 
$$
\chi \in \mathbf L\otimes A := \Hom(\mathbf L^\vee,A),
$$
such that it induces a commutative diagram

\

\centerline{
\xymatrix{
\mathbf L^\vee= \mathbf L_{\le 0}^\perp \ar@{>>}[d]& \supsetneq &  \mathbf L_{\le 1}^\perp  \ar@{>>}[d]& \cdots & \supsetneq &  \mathbf L_{\le r }^\perp \ar@{>>}[d] \\
A= G_{\le 0}^\perp                                                    & \supsetneq &  G_{\le 1}^\perp                                 &  \cdots & \supsetneq &   G_{\le r }^\perp
}
}

\

\noindent
i.e., 
$$
G_{\le i}^\perp = \chi(\mathbf L_{\le i}^\perp), \quad i=0,\ldots, r-1.
$$ 
Such character $\chi$ is surjective (case $i=0$) and induces surjective homomorphisms
$$
\chi_i : \mathbf L_i^\vee \ra A_i=\Hom(G_i), \quad i=1,\ldots, r,
$$
where $\mathbf L_i = \mathbf L_{\le i} /  \mathbf L_{\le i-1}$ and $G_i= G_{\le i} / G_{\le i-1}$. 
Again, such $\chi$ defines an elementary matrix with indices
$$
\chi, (\chi_1,\ldots, \chi_r),
$$
taking the sum over all such $\chi$ we obtain the desired homomorphism. 
\end{itemize}

\

\begin{prop}
\label{prop:identy}
Using the identifications
$$
\mathcal M_n^-(G)\otimes \bQ= H_n^{BM}(\mathbb X_n, \mathcal L_{n,G})
$$
and formulas \eqref{eqn:del} and \eqref{eqn:del2} we obtain homomorphisms
$$
\mathcal M_{n_1}^-(G_1)\otimes \cdots \otimes \mathcal M_{n_r}^-(G_r) \otimes \bQ \rightleftarrows \mathcal M_n^-(G)\otimes \bQ
$$
which are the same as those induced from $\Delta$ and $\nabla$ in Section~\ref{sect:co}. 
\end{prop}

\begin{proof}
The case of the product follows immediately from the definition: 
a basis $e_1,\ldots, e_n $ of $\mathbf L$ gives a closed Borel-Moore chain $\simeq \bR_{>0}^n$, consisting of diagonal forms $h$ in this basis. 

To verify the co-product we need the following: 
let $\mathbf L\simeq \bZ^n$ be the standard coordinate lattice, up to the action of 
$\mathfrak S_n \ltimes (\bZ/2\bZ)^n$ interchanging the coordinates and 
acting by sign on each coordinate. 
We have a canonical Borel-Moore closed chain  
$$
C_n\subset \mathrm{Chains}_n^{BM}(\mathbb X_n,\bZ), \quad \partial (C_n) =0, 
$$
given by the image of positive diagonal matrices. Given a flag
$$
0=\mathbf L_{\le 0} \subsetneq \cdots \subsetneq \mathbf L_{\le r}  = \mathbf L
$$
and using the correspondence 
$$
\mathbb Y^{\Delta}_{n_1,\ldots, n_r}
$$
we obtain a closed Borel-Moore chain 
$$
C_{\mathbf L_\bullet} \subset \mathrm{Chains}^{BM}_n(\mathbb X_{n_1}\times \cdots \times \mathbb X_{n_r},\bZ),
$$
to any point $h$ in $C_n$ we associate a collection 
$$
(h_1,\ldots, h_r)\in \mathbb X_{n_1}\times \cdots \times \mathbb X_{n_r}.
$$ 

The main observation is that if the flag is not 
compatible with the chosen coordinate decomposition, then the corresponding chain is a boundary. From this it follows that only 
the coordinate flags contribute to the formula. 
\end{proof}

\

Following the reasoning in Section~\ref{sect:co}, specifically \eqref{eqn:prim}, we define
$$
H_{n,prim}^{BM}(\mathbb X_n,\mathcal L_{n,G})\subseteq H_{n}^{BM}(\mathbb X_n,\mathcal L_{n,G})
$$
as the common kernel of all nontrivial co-multiplication homomorphisms ($r\ge 2$). 
Evidently, we have
$$
\mathcal M_{n,prim}^-(G)\otimes \bQ=H_{n,prim}^{BM}(\mathbb X_n,\mathcal L_{n,G}),
$$
under the above identifications.

We recall that
$$
H_{n,int}(\mathbb X_n,\mathcal L_{n,G}):=
\mathrm{Image}\left( H_n(\mathbb X_n,\mathcal L_{n,G}) \ra 
H^{BM}_n( \mathbb X_n,\mathcal L_{n,G})\right).
$$

\begin{conj}
\label{conj:222}
For every nontrivial finite abelian group $G$ and every $n\ge 1$,
we have 
$$
H_{n,cusp}(\mathbb X_n,\mathcal L_{n,G})  = H_{n,int}(\mathbb X_n,\mathcal L_{n,G}) = 
H_{n,prim}^{BM}(\mathbb X_n,\mathcal L_{n,G}).
$$
\end{conj}

This conjecture is essentially our guess, stated implicitly in Section~\ref{sect:co}.
Assuming this conjecture, we would obtain the following reformulation:

\begin{conj}
\label{conj:flags}
For every nontrivial finite abelian group $G$ and every $n\ge 1$, the natural homomorphism 
$$
\bigoplus_{r=1}^n\bigoplus_{\substack{n_1+\cdots +n_r = n\\\mathcal G_\bullet \text{ of length $r$}}} 
H_{n_1,cusp}(\mathbb X_{n_1}, \mathcal L_{n_1,gr_1(\mathcal G_\bullet)}) \otimes \cdots \otimes H_{n_r,cusp}(\mathbb X_{n_r}, \mathcal L_{n_r,gr_r(\mathcal G_\bullet)}) \ra
$$
$$
\ra  H_{n}^{BM}(\mathbb X_{n}, \mathcal L_{n,G})
$$
is an isomorphism.
\end{conj}

Representation theory gives a canonical splitting of cohomology of arithmetic groups into the sum of the cuspidal and 
the remaining (Eisenstein) parts, after tensoring by $\bC$. Our considerations, for $\mathrm{GL}_n(\bZ)$, suggest that we have a splitting over $\bQ$. Namely, 
define 
$$
H_{n,coprim}^{BM}(\mathbb X_{n}, \mathcal L_{n,G}) 
$$
as the quotient by the sum of images of all nontrivial product maps \eqref{eqn:del}. It is tempting to make a companion conjecture:

\begin{conj}
\label{conj:comp}
For every nontrivial finite abelian group $G$ and every $n\ge 1$,
the homomorphism 
$$
H_{n}^{BM}(\mathbb X_n,\mathcal L_{n,G}) \ra
$$
$$
\ra  \bigoplus_{r=1}^n\bigoplus_{\substack{n_1+\cdots +n_r = n\\\mathcal G_\bullet \text{ of length $r$}}} 
H_{n_1,coprim}(\mathbb X_{n_1}, \mathcal L_{n_1,gr_1(\mathcal G_\bullet)}) \otimes \cdots \otimes H_{n_r,coprim}(\mathbb X_{n_r}, \mathcal L_{n_r,gr_r(\mathcal G_\bullet)}) 
$$
is an isomorphism. 
\end{conj}

\begin{conj}
The composition
$$
H_{n,prim}^{BM}(\mathbb X_n,\mathcal L_{n,G}) \hookrightarrow H_{n}^{BM}(\mathbb X_n,\mathcal L_{n,G}) \twoheadrightarrow H_{n,coprim}^{BM}(\mathbb X_n,\mathcal L_{n,G}) 
$$
is an isomorphism. 
\end{conj}

The considerations above fit into a general framework. For $n\ge 1$, let 
$R_n$ be the set of finite-dimensional irreducible representations of $\mathrm{GL}_n(\bZ)$ which appear as direct summands of tensor products of 
\begin{itemize}
\item 
representations of
$$
\mathrm{GL}_n(\hat{\bZ}) = \prod_p \mathrm{GL}_n(\bZ_p) 
$$
\item
irreducible algebraic representations
$$
\rho_{\lambda} : \mathrm{GL}_n(\bQ)\ra \mathsf V_{\lambda}
$$
with highest weight $\lambda$. 
\end{itemize}
Obviously, $R_1$ consists of two elements, and $R_n$ are countable infinite sets for $n\ge 2$. 

Given 
$$
\rho_1\in R_{n_1},\rho_2\in R_{n_2},\rho \in R_{n}, \quad \text{ for } n=n_1+n_2,
$$ 
we can define the multiplicity space 
$$
\mathrm{mult}_{\rho_1,\rho_2}^{\rho} \in \mathrm{Vect}_{\bC},
$$
a finite-dimensional complex vector space, by 
$$
\Hom_{\mathrm{GL}_{n_1}(\bZ)\times \mathrm{GL}_{n_1}(\bZ)} (\rho_{n_1} \boxtimes \rho_{n_2},\rho_{|_{\mathrm{GL}_{n_1}(\bZ)\times \mathrm{GL}_{n_2} (\bZ)}}).
$$
The correspondence 
$$
\mathbb Y^\nabla_{n_1,n_2}
$$ 
gives rise to a natural homomorphism
$$
\mathrm{mult}_{\rho_1,\rho_2}^{\rho} \otimes H^{BM}_{*}(\mathbb X_n, \rho_{n_1}) \otimes H^{BM}_{*}(\mathbb X_{n_2}, \rho_{n_2}) 
\ra H^{BM}_{*}( \mathbb X_n, \rho).
$$
The collection of these can be organized in the following way:  let $\mathcal C$ be a semi-simple (in the countable sense) $\bC$-linear tensor category, 
with countable sums and tensor products commuting with sums, and with simple objects $\epsilon_{\rho}$, corresponding to 
$\rho\in \coprod_{n\ge 1} R_n$; the tensor product given by
$$
\epsilon_{\rho_1} \otimes \epsilon_{\rho_2} = \oplus_{\rho}  \,\,\, \mathrm{mult}_{\rho_1,\rho_2}^{\rho} \otimes_{\bC} \epsilon_{\rho}.
$$
The expression on the right is infinite. Put
$$
\mathcal A_\bullet:=\bigoplus_{n\ge 1} \bigoplus_{\rho\in R_n} H_\bullet^{BM}(\mathbb X_n,\rho\otimes \epsilon_\rho) \in Ob(\mathcal C)
$$
The object 
$\mathcal A_\bullet$ carries the structure of a super-commutative associated $\bZ$-graded nonunital algebra in $\mathcal C$. 
Using chains instead of homology groups gives rise to a  commutative differential $\bZ$-graded nonunital algebra which by Koszul duality can be identified 
with a differential graded Lie algebra (or $\mathsf L_{\infty}$-algebra). The next question is: what is this algebra, or its Koszul dual dg Lie algebra?

The category $\mathcal C$ itself seems to have a description as a category 
of representation of a certain type of an infinite-dimensional semi-group.

In the model example,  consider $R_n^{fin}$, consisting of irreducible representations of the symmetric group $\mathfrak S_n$. 
Then the corresponding analog $\mathcal C^{fin}$ of the category $\mathcal C$
is the a subcategory of Deligne's category of representations of $\mathfrak{gl}_{t}$, where $t$ is a parameter (fractional dimension). 

In the second model example, more relevant to our considerations, 
let $R_{n}^{alg}$ be the set of irreducible algebraic representations
$$
\rho_{\lambda} : \mathrm{GL}_n(\bQ)\ra \mathsf V_{\lambda}
$$
with highest weights $\lambda$. 
Defining multiplicity spaces $\mathrm{mult}_{\rho_1,\rho_2}^\rho$ in a similar fashion, we obtain 
a category 
$\mathcal C^{alg},$ 
which is the category of highest weight representations of the (well-known) central extension 
$$
1\ra \bC^\times \ra \mathsf G\ra \mathrm{GL}_\infty(\bC)^\circ\ra 1,
$$
where $\mathrm{GL}_\infty(\bC)^\circ$ is the connected component of the identity of 
the group 
$$
\{ g\in \Aut_{cont, \bC-mod}(\bC^\infty)\}, \quad \text{ where } \mathbb C^\infty := \bC((t)).
$$
The group $\mathsf G$ acts on a space of countable dimension
$$
\mathsf V := \oplus_{i\in \bZ} \wedge^{\frac{\infty}{2} +i} (\mathbb C^\infty). 
$$
An analog of Weyl-Schur duality says that, for all $n\ge 1$, 
$\mathrm{GL}_n(\bC)$ acts on $\mathsf V^{\otimes n}$, commuting with $\mathsf G$-action, and identifying 
highest weight representations of $\mathsf G$ of level $n$ (i.e., those where the central extension acts with character $z\mapsto z^n$) with 
algebraic irreducible representations of $\mathrm{GL}_n(\bC)$. 

From our perspective, it would be important to identify explicitly the category $\mathcal C^{/p}$, whose simple objects 
correspond to irreducible finite-dimensional representations of the groups $\mathrm{GL}_n(\bF_p)$, $n\ge 1$, and the category 
 $\mathcal C^{p}$, whose simple objects 
correspond to irreducible finite-dimensional continuous 
$\mathrm{GL}_n(\bZ_p)$, $n\ge 1$.

We can develop a similar framework for co-multiplication. 
Given 
$$
\rho_1\in R_{n_1},\rho_2\in R_{n_2},\rho \in R_{n}, \quad \text{ for } n=n_1+n_2,
$$ 
we can define the co-multiplicity space 
$$
\mathrm{comult}^{\rho_1,\rho_2}_{\rho} \in \mathrm{Vect}_{\bC},
$$
a finite-dimensional complex vector space, as
$$
\Hom_{P_{n_1,n_2} (\bZ)} (\rho_{|_{P_{n_1,n_2} (\bZ)}}, \rho_{n_1} \boxtimes \rho_{n_2}),
$$
where 
$$
P_{n_1,n_2}\subset \mathrm{GL}_{n_1}
$$ 
is the stabilizer of the flag $\bZ^{n_1} \subset \bZ^n$. 
The correspondence 
$$
\mathbb Y^{\Delta}_{n_1,n_2}
$$
gives rise to a natural homomorphism
$$
\mathrm{comult}_{\rho_1,\rho_2}^{\rho} \otimes H^{BM}_{*}( \mathbb X_n, \rho) \ra 
H^{BM}_{*}(\mathbb X_n, \rho_{n_1}) \otimes H^{BM}_{*}(\mathbb X_{n_2}, \rho_{n_2}).
$$
We obtain a co-associative co-algebra, without a unit, in a tensor category which is no longer symmetric, a priori. 

Note that there might be nontrivial extensions between two representations from $R_n$, 
which suggests that the definition of the category $\mathcal C$ and algebra $\mathcal A_\bullet$ could be enhanced by considering extension data. 
Also, the category $\mathcal C$ is not rigid, and hence should interpreted not as a category of representations of a group but rather of a semi-group.

Finally, all considerations above can be carried over to the number field case, but in this case, instead of lattices we should consider
all nontrivial finitely-generated torsion-free modules.

\ 

\section{Case $n=2$: modular symbols}
\label{sect:modular}

We recall the definition of modular symbols of weight 2 for 
$$
\Gamma_1(N){:=\left\{\gamma \in \mathrm{SL}_2(\bZ): \gamma=
\left(\begin{array}{cc} 1 & * \\ 0 & 1
\end{array}
\right)\mod N
\right\}}, \quad N\in \bZ_{\ge 2}.
$$
Let $\mathbb M_2(\Gamma_1(N))$ be the $\bQ$-vector space generated by pairs $(c,d)$ with 
$$
c,d \in \bZ/N, \quad \gcd(c,d,N)=1,
$$
and
subject to relations
\begin{enumerate}
\item[(1)] $(c,d)=-(d,-c) \text{  (and hence }=(-c,-d)=-(-d,c) \,)$,
\item[(2)]  $(c,d)+(d,-c-d)+(-c-d,c)=0$.
\end{enumerate}
It is known that  $\mathbb M_2(\Gamma_1(N))$ is naturally identified with Borel-Moore homology group $H_1^{BM}({X_1(N)}, \bQ)$ of the complex 
 modular curve
$$
X_1(N):=\Gamma_1(N)\backslash \mathcal H,
$$
where $\mathcal H$ is the upper half-plane. The symbol $(c,d)$ corresponds to the image in  ${X_1(N)}$ of the geodesic path from $\bf a/\bf c$ to $\bf b/\bf d$, where 
$$\left(\begin{array}{cc} \bf a & \bf b \\ \bf c & \bf d
\end{array}
\right)
 \in \Gamma_1(N)$$
is any element with $c,d=\bf c, \bf d \mod \it N$.

Using (1) we can rewrite (2) as 
\begin{enumerate}
\item[($2'$)]  $(d,c)=(d,c-d)+(d-c,c)$.
\end{enumerate}
Indeed, substituting $c\mapsto -c$ into (2), we obtain
\begin{align*}
0 & \stackrel{(2)}{=}(-c,d)+(d,c-d)+(c-d,-c) \\
    & \stackrel{(1)}{=} - (d,c) + (d,c-d)+(c-d,-c)\\
    &  \stackrel{(1)}{=} - (d,c) + (d,c-d)+(d-c,c)
\end{align*} 



There is an involution on  $\mathbb M_2(\Gamma_1(N))$
$$
\iota: (c,d) \mapsto (-c,d)\stackrel{(1)}{=} -(d,c).
$$
Written in the form $(c,d)\mapsto -(d,c)$ it obviously preserves relations ($2'$) and cyclic anti-symmetry $(1)$.
It corresponds to the automorphism  of  the first homology group coming from the anti-holomorphic involution on $X_1(N)$ associated with the map $\tau\mapsto -\bar{\tau}, \tau \in \mathcal H$, on the universal cover.
Denote by $\mathbb M_2^-(\Gamma_1(N))$  the $(-)$-eigenspace for the involution $\iota$. 
 
\

The dimensions are given by 
$$
\dim(\mathbb M_2(\Gamma_1(N)))=2g+C(N)-1, \, 
\dim(\mathbb M_2^-(\Gamma_1(N)))=g+\frac{C(N)-C_2(N)}{2},
$$
where
\begin{itemize}
\item 
 $g=g(N)$ is the genus of $\overline{X_1(N)}$, which is the same as the dimension of the 
space of cusp forms of weight 2 for $\Gamma_1(N)$ (see the table in Section~\ref{sect:co}),
\item 
$C(N)$ is the number of cusps (elements of $\bP^1(\bQ)/\Gamma_1(N)$), and 
\item $C_2(N)$ is the number of cusps fixed by the anti-holomorphic involution described above.
\end{itemize}

For $N=1,2,3,4$, $C(N)=C_2(N)=1,2,2,3$, respectively; and for 
$N\ge 5$, the cardinalities $C(N),C_2(N)$ are given by
$$
C(N)=\frac{1}{2}\sum_{d|N} \phi(d)\,\phi(N/d),
$$
$$C_2(N)=\left\{ \begin{array}{ll} \phi(N) +\phi(N/2)& \text{ if } N \text{ is even,}\\
\phi(N)& \text{ if } N \text{ is odd.}\end{array}\right.$$

\ 

 Now we discuss the relation to our groups $\mathcal M_2(\bZ/N\bZ)$ and $\mathcal M_2^-(\bZ/N\bZ)$.

\begin{prop}\label{prop:iso2}
$\mathcal M_2^-(\bZ/N \bZ) \otimes \bQ$ is  isomorphic to $\mathbb M_2^-(\Gamma_1(N))$ .
\end{prop}
\begin{proof}
Indeed, the subspace $\mathbb M_2^-(\Gamma_1(N))$ (or, better, quotient space) can be described in terms of generators and relations as
\begin{enumerate}
\item[(R1)] $(a_1,a_2)^-=(a_2,a_1)^-$
\item[(R2)] $(a_1,a_2)^-=(a_1,a_2-a_1)^- +(a_1-a_2,a_2)^-$
\item[(R3)] $(a_1,a_2)^-=-(a_2,-a_1)^-$
\end{enumerate}

Here (R3) is the same as (1), (R2) is the same as (2'), and (R1) is $\iota$-invariance.
Therefore,  the natural map 
 $$\mathcal M_2^-(\bZ/N \bZ) \otimes \bQ \stackrel{\sim}{\lra}   \mathbb M_2^-(\Gamma_1(N)),\quad  \langle a_1,a_2\rangle ^-   \mapsto  (a_1,a_2)^-$$
is an isomorphism, as relations (R1),(R2),(R3) are exactly the defining relations for $\mathcal M_2^-(\bZ/N \bZ)$. \end{proof}

\

Note that 
$$
(a,0)^-=(0,a)^- = 0\in \mathbb M_2^-(\Gamma_1(N)),
$$
by (R1) and (R3). 
Incidentally, relation (R2) can be replaced by the co-vector version
\begin{enumerate}
\item[$(\rm R2^*)$] $(a_1,a_2)^-=(a_1+a_2,a_2)^- +(a_1,a_1+a_2)^-$
\end{enumerate}
Indeed, substitute $a_1\mapsto a_1, a_2\mapsto a_1+a_2$ into relation (R2) and use dihedral symmetry by (R1) and (R3).

\

As a corollary of Conjectures \ref{conj:1} and \ref{conj:2}, together with the guesses
$$
\dim(\mathcal M_{2,prim}(\bZ/N\bZ)\otimes\bQ)=  \dim(\mathcal M_{2,prim}^-(\bZ/N\bZ)\otimes\bQ)=g(N),
$$
we would obtain a formula which follows from Proposition 
\ref{prop:iso2}:
\begin{align*} \dim (\mathcal M_2^-(\bZ/N \bZ)\otimes \bQ)  {=}   & g(N)+\frac{1}{4}\sum_{d|N,3\le d\le N/3}\phi(d)\,\phi(N/d)\\
\stackrel{\text{for all }N\ge 1}{ =} \dim(\mathbb M_2^-(\Gamma_1(N)))=& g(N)+\frac{C(N)-C_2(N)}{2} 
\end{align*}
and a \emph{hypothetical} formula: 
\begin{align*}
\dim (\mathcal M_2(\bZ/N \bZ)\otimes \bQ)  \stackrel{?}{=} & \,\,g(N)+\frac{1}{2}\sum_{d|N,d\ge 3}\phi(d)\,\phi(N/d) \\
\stackrel{\text{for }N\ge 5}{ =} & \,\,g(N)+C(N)-\frac{C_2(N)}{2} 
\end{align*}

Presumably, one can deduce the above formula using the relation between the Steinberg module and the module $\mathcal F_2$
(see  Proposition~\ref{prop:stein}).
The formulas for dimensions simplify when $N=p\ge 5$ is a prime:
$$g(p)=\frac{(p-5)(p-7)}{24}, \quad C(p)=C_2(p)=p-1,$$
$$\dim(\mathcal M_2^-(\bZ/p \bZ)\otimes \bQ)=\dim(\mathbb M_2^-(\Gamma_1(p)))=g(p)
$$
\begin{equation}\label{eq:p_dim2}
\dim(\mathcal M_2(\bZ/p \bZ)\otimes \bQ)\stackrel{?}{=} \frac{p^2+23}{24}=g(p)+\frac{p-1}{2}.
\end{equation}

\

The rest of the section will be devoted to a direct proof of  \eqref{eq:p_dim2}.

 We have two maps
\begin{align}\label{eq:easyepi}
\mathcal M_2(\bZ/p \bZ)&\twoheadrightarrow \mathcal M_2^-(\bZ/p \bZ),\quad \langle a,b\rangle\mapsto \langle a,b\rangle^-\\
\label{eq:deltaepi}
\mathcal M_2(\bZ/p \bZ)&\stackrel{\Delta}{\rightarrow} \mathcal M_1(1)\otimes 
\mathcal M_1^-(\bZ/p \bZ)=\mathcal M_1^-(\bZ/p \bZ),
\end{align}
where the second map \eqref{eq:deltaepi}
is the (only possible) co-product map  given by
$$\langle a,b \rangle \mapsto (1-\delta_{a,0})\langle a \rangle^-+(1-\delta_{b,0})\langle b \rangle^-.$$
The first map \eqref{eq:easyepi} is  surjective by definition, and second map \eqref{eq:deltaepi} is surjective up to 2-torsion: its right inverse after tensoring with $\bQ$ is given by 
\begin{equation}\label{eq:rinverse}
\langle a \rangle^-\mapsto \frac{1}{2}\big(
\langle a,0\rangle -\langle -a,0 \rangle\big)
\end{equation}

The validity of  formula \eqref{eq:p_dim2}
 follows from the following result:

\begin{prop}
\label{prop:m2}
The map given by the sum of  \eqref{eq:easyepi} and \eqref{eq:deltaepi}:
$$
\mathcal M_2(\bZ/p\bZ)\to  \mathcal M_2^-(\bZ/p \bZ)\oplus \mathcal M_1^-(\bZ/p \bZ)
$$ 
is an isomorphism up to torsion.
\end{prop}

\begin{proof}

We will check (after tensoring with $\bQ$) that the kernel of \eqref{eq:easyepi} is generated by the image of \eqref{eq:rinverse}.

By definition \eqref{eqn:anti}, the kernel of \eqref{eq:easyepi} is spanned by  elements 
\begin{equation}\label{eq:genker}\langle a,b\rangle  + \langle a,-b\rangle  \in\mathcal M_2(\bZ/p\bZ).\end{equation}

\begin{lemm}
\label{lemm:below2}
For all $a,b,\in \bZ/p\bZ$, with $a\neq 0$, we have 
\begin{equation*}
\label{eqn:app}
\langle a,b\rangle  + \langle a,-b\rangle  = 2\cdot \langle a,0\rangle \in \mathcal M_2(\bZ/p\bZ).
\end{equation*}
\end{lemm}

\begin{proof}

We  have from (M):
\begin{align*}
\langle a,b\rangle  & = \langle a-b,b\rangle  + \langle a,b-a\rangle\\
\langle a-b,a\rangle & = \langle -b,a\rangle  + \langle a-b, b\rangle.
\end{align*}
Taking the difference between the first and the second line, we obtain
$$
\langle a,b\rangle  + \langle -b,a\rangle  =  \langle a,b-a\rangle  + \langle a,a-b\rangle,
$$
which we can write, by (S), as
$$
\langle a,b\rangle  + \langle a,-b\rangle  =  \langle a,b-a\rangle  + \langle a,-b+a\rangle.
$$
Iterating this, we get
$$
\langle a,b\rangle  + \langle a,-b\rangle  =  \langle a,b-ma\rangle  + \langle a,-b+ma\rangle, \quad m=1,\ldots, p. 
$$
For $a\neq 0 \pmod p$, there is an $m$ solving the equation $ma=b\pmod p$, which implies the claimed identity
\begin{equation}
\label{eqn:app}
\langle a,b\rangle  + \langle a,-b\rangle  = 2\cdot \langle a,0\rangle. 
\end{equation}
\end{proof}

\begin{coro} \label{cor:easyker} The kernel of \eqref{eq:easyepi} is spanned by elements
$$2\cdot\langle a,0\rangle,\quad \langle a,0\rangle+\langle-a,0\rangle\in  \mathcal M_2(\bZ/p\bZ)\qquad (a\ne 0\in \bZ/p\bZ).$$
\end{coro}
\begin{proof} The case $a\ne 0$ in \eqref{eq:genker} is covered by  Lemma \ref{lemm:below2}. If $a=0$ then we rename variable $b\rightsquigarrow a$ and apply symmetry relation (S).
\end{proof}
\begin{lemm}
\label{lemm:below}
For all $a\in \bZ/p\bZ, a\neq 0$, we have
$$
\langle a,0\rangle + \langle -a,0\rangle =0 \in \mathcal M_2(\bZ/p\bZ)\otimes \bQ.
$$
\end{lemm}

\begin{proof}
Replacing $a$ by $-a$ in \eqref{eqn:app} and adding the equations, we obtain
$$
\big(\langle a,b\rangle +   \langle -a,b\rangle\big)  + \big(\langle a,-b\rangle      + \langle -a,-b\rangle\big)  = 2\cdot \big( \langle a,0\rangle +  \langle -a,0\rangle\big).
$$ 
Using  again \eqref{eqn:app}, with $a$ replaced by $b$, respectively $-b$, we find
\begin{equation}
\label{eqn:22}
2\cdot \big( \langle b,0\rangle +  \langle -b,0\rangle\big) = 2\cdot \big( \langle a,0\rangle + \langle -a,0\rangle\big), 
\end{equation}
for all $a,b\neq 0$. 
To show the vanishing of 
$$
{\bm{\delta}}:=
\langle 1,0\rangle + \langle -1,0\rangle  \in \mathcal M_2(\bZ/p\bZ)\otimes \bQ
$$
consider the sum
\begin{equation*}
    \sum_{a,b \neq 0}  \big(\langle a,b \rangle  + \langle b, -a \rangle \big)= 2(p-1) \cdot \sum_{b\neq 0}  \langle b,0\rangle = 
(p-1)^2\bm{\delta},
\end{equation*}
here we substituted \eqref{eqn:app} and \eqref{eqn:22}.
Apply the blow-up relation (M) to each term and relate
  to the original sum:
  \begin{multline*}
   \stackrel{({\mathrm M})}{=} \sum_{a,b \neq 0} \langle  a - b, b  \rangle  +     \sum_{a,b \neq 0} \langle a , b-a  \rangle   +    \sum_{a,b \neq 0}
    \langle  b +a, -a  \rangle  +        \sum_{a,b \neq 0} \langle b, -a - b \rangle  =\\
   \stackrel{({\mathrm S})} {=}    
    4\sum_{b \neq 0, a \neq -b} \langle  a , b  \rangle  =\\
    =4\sum_{a,b \neq 0} \langle  a , b  \rangle+4\sum_{a\ne 0}\langle a,0\rangle-4\sum_{a\ne 0}\langle a,-a\rangle=\\ 
 =2 (p-1)^2\, \bm{\delta}  + 2 (p-1)
 \,\bm{\delta}
 =2p(p-1) \,\bm{\delta}\\ 
   \end{multline*}
After the blow-up relation, we changed the variables in the summation using symmetry relation,
 then related to the original range of the summation with discrepancy terms, and 
 used the relations
\begin{equation*}
   \sum_{a \neq 0} \big(\langle a, 0 \rangle  + \langle -a,0\rangle\big)  = (p-1) \,\bm{\delta}
 \end{equation*}
 and
 $$
 \langle a, -a\rangle =0\qquad \impliedby
 \qquad \langle a,0\rangle \stackrel{(\mathrm M)}{=} \langle a,0\rangle + \langle a,-a\rangle.
 $$ 
Finally, we obtain
\begin{equation*}
  (p-1)^2  \,\bm{\delta} = 2 p (p-1)\, \bm{\delta},
\end{equation*}
which implies
\begin{equation}
\label{eqn:psq}
  (p^2 - 1) \, \bm{\delta}  = 0  \in \mathcal M_2(\bZ/p\bZ).
\end{equation}
It follows that for all $a\neq 0$ we have the claimed identity
\begin{equation*}
\label{eqn:aa}
 \langle a,0\rangle + \langle -a,0\rangle =0 \in \mathcal M_2(\bZ/p\bZ)\otimes \bQ.
\end{equation*}
\end{proof}

Now we are ready to finish the proof of Proposition  \ref{prop:m2}.  By Corollary   \ref{cor:easyker} and Lemma \ref{lemm:below}, the kernel of \eqref{eq:easyepi} is spanned (up to torsion) by elements of the form $\langle a,0 \rangle$. Using again  Lemma \ref{lemm:below} we see that these elements can be written as
$$ \langle a,0 \rangle=   \frac{1}{2}\big(
\langle a,0\rangle -\langle -a,0 \rangle\big)\in \mathcal M_2(\bZ/p\bZ)\otimes \bQ $$
 Therefore, we get exactly the image of the right inverse \eqref{eq:rinverse}.
\end{proof}

\begin{rema}
The factor $(p^2-1)$  in \eqref{eqn:psq} gives a partial explanation for the experimentally observed jumping behavior of 
$\dim(\mathcal M_2(\bZ/p\bZ)\otimes \mathbb F_{\ell})$, for primes $\ell \mid (p\pm 1)$, 
see Section~\ref{sect:exp}.
\end{rema}

\section{Experiments}
\label{sect:exp}

Here we present results of numerical experiments, 
performed using a fast linear algebra solver \cite{spasm}. We computed dimensions of 
$$
\mathcal B_n(\bZ/N\bZ), \quad \mathcal M_n(\bZ/N\bZ)
$$
over $\bQ$ and various finite fields.  
The size of the (very sparse) matrices grows as $\sim N^n$. For example, for $n=5$ and $N=81$, the  part of constraints corresponding to $k=2$ in (B) or 
(M), gives $\sim 3\cdot 10^8$ equations on $\sim 3\cdot 10^7$ variables, with $\sim 10^9$ non-zero coefficients. This overdetermined system has a unique (up to scalar) nontrivial solution in $\bQ$. The calculation takes about 4 hours. 

\

Numerically, we found:
\begin{itemize}
\item For $p$ a prime, 
$$
\dim(\mathcal B_2(\bZ/p\bZ)\otimes {\bQ}) = \frac{p^2-1}{24} +1= \frac{p^2+23}{24},
$$
while the difference
$$
\Delta_{2,\ell}(\bZ/p\bZ):=\dim(\mathcal B_2(\bZ/p\bZ)\otimes {\mathbb F_\ell}) -  \frac{p^2+23}{24} 
$$
varies significantly; there are frequent jumps when $\ell\mid  (p\pm 1)$, e.g., 
$$
\Delta_{2,31} (\bZ/61\bZ)=1.
$$
\item 
For $p$ a prime, 
$$
\Delta_{3,\bQ}(\bZ/p\bZ) :=\dim(\mathcal B_3(\bZ/p\bZ)\otimes {\bQ}) -\frac{(p-5)(p-7)}{24}  = 0
$$
for all primes up to 41, but 
$$
\Delta_{3,\bQ}(\bZ/p\bZ) = 1, \quad \text{ for } \, p=43, 59, \ldots
\,.
$$
\item The difference
$$
\Delta_{3,\ell}(\bZ/p\bZ):=\dim(\mathcal B_3(\bZ/p\bZ)\otimes {\mathbb F_{\ell}}) -  \frac{(p-5)(p-7)}{24} 
$$
also jumps for many $\ell \mid (p\pm 1)$. 
\item For all primes $p$ up to 41 we have $\dim(\mathcal B_4(\bZ/p\bZ)\otimes \bQ)=0$, but
$$
 \dim(\mathcal B_4(\bZ/p\bZ)\otimes \bQ)=1,  \text{ for }  p= 43, 59, \ldots\,.
$$
\end{itemize}

On the next page we present a more systematic table of dimensions. 
All dimensions, for $\bQ$-coefficients, are compatible with the conjectures in Section~\ref{sect:co}. 
The items in bold indicate the smallest $N$ for which the rank is positive.

\

\begin{itemize}
\item $\dim(\mathcal B_n(\bZ/N\bZ)\otimes \bQ)=\dim(\mathcal M_n(\bZ/N\bZ)\otimes \bQ)$ for $n=2,3$:
\end{itemize}          
\

\small{\begin{tabular}{c|c|c|c|c|c|c|c|c|c|c|c|c|c|c|c|c|c|c}
$N$    &  2  & 3 &  4 & 5 & 6 & 7 & 8 & 9 & 10 & 11 & 12 & 13 & 14 & 15 & 16 & 17 & 18  \\
\hline
\hline
\rule{0pt}{3ex}
$n$=2 & 0 & \bf{1} & 1  &  2 & 2 & 3 & 3 & 5 & 4 &6& 7 &8 & 7 &13 &10 &13 & 12\\
\rule{0pt}{2ex}
$n$=3 & 0 & 0 & 0  & 0 & 0 & 0&0& \bf{1} & 0 &1 & 2 & 2 &1 & 5 &3 &5& 5
\end{tabular}
}

\

\small{
\begin{tabular}{c|c|c|c|c|c|c|c|c|c|c|c|c|c|c|c|c|c|c}
$N$  &19&20&21&22&23&24 &25& 26&27&28&29&\dots &180 &181\\
\hline
\hline
\rule{0pt}{3ex}
$n$=2 & 16&17&23&16&23&23&30& 22& 34 & 31& 36 &\dots &   989& 1366 \\
\rule{0pt}{2ex}
$n$=3 & 7 & 7 &11& 7 &12&13&16&12&21&17&22&\dots& 1740 &1276
\end{tabular}
}

\

\

\begin{itemize}
\item 
\normalsize{$\dim(\mathcal B_n(\bZ/N\bZ)\otimes \bQ)=\dim(\mathcal M_n(\bZ/N\bZ)\otimes \bQ)$ for $n=4$:}
\end{itemize}
\

\small{
\begin{tabular}{c|c|c|c|c|c|c|c|c|c|c|c|c|c|c|c|c|c|c}
$N$ &27&28&29&30&31&32&33&34&35&36 &\dots&105&106&107\\
\hline
\hline
\rule{0pt}{3ex}
$n$=4 & \bf{1} & 0 & 0 & 0&0&0&2&0 &0&3 &\dots& 114 & 0 &3
\end{tabular}
}

\

\

\begin{itemize}
\item 
\normalsize{$\dim(\mathcal M_{4,prim}^-(\bZ/N\bZ)\otimes \bQ) =0$ for $N\le 242$:}
\end{itemize}
\

\begin{itemize}
\item 
\normalsize{$\dim(\mathcal B_n(\bZ/N\bZ)\otimes \bQ)=\dim(\mathcal M_n(\bZ/N\bZ)\otimes \bQ)$ for $n=5$:}
\end{itemize}

\

\small{
\begin{tabular}{c|c|c|c|}
$N$ &\dots$\le$ 80&81&82\\
\hline
\hline
\rule{0pt}{3ex}
$n$=5& 0 & \bf{1} & 0 
\end{tabular}
}

\

\begin{itemize}
\item 
\normalsize{$\dim(\mathcal B_n(\bZ/N\bZ)\otimes \bF_2)$ and $\dim(\mathcal M_n(\bZ/N\bZ)\otimes \bF_2)$ for $n=2,3,4,5$:}
\end{itemize}

\

\small{
\begin{tabular}{c|c|c|c|c|c|c|c|c|c|c|c|}
$N$    &  2  & 3 &  4 & 5 & 6 & 7 & 8 & \dots & 16 & \dots  & 32 \\
\hline
\hline
\rule{0pt}{3ex}
$\mathcal{B}_2$  &0 &\bf{1}&1&2&3&4&4&\dots &13&\dots& 44\\
\rule{0pt}{2ex}
$\mathcal{M}_2$ &\bf{1}&2&3&5&5&8&8&\dots & 21& \dots  &60\\
\hline
\rule{0pt}{3ex}
$\mathcal{B}_3$ &0&0&0&0&0&\bf{1}&1&\dots&8 &\dots& 43\\
\rule{0pt}{2ex}
$\mathcal{M}_3$  &0&0&\bf{1}&1&3&2&5&\dots &  21 &\dots &87  \\
\hline
\rule{0pt}{3ex}
$\mathcal{B}_4$ & 0&0&0&0&0&0&0&\dots &\bf{1} &\dots&12\\
\rule{0pt}{2ex}
$\mathcal{M}_4$  &0&0&0&0&0&0&\bf{1}&\dots &9 &\dots &  55 \\
\hline
\rule{0pt}{3ex}
$\mathcal{B}_5$ & 0&0&0&0&0&0&0&\dots & 0 &\dots&\bf{1}\\
\rule{0pt}{2ex}
$\mathcal{M}_5$  &0&0&0&0&0&0& 0&\dots &\bf{1} &\dots &  13
\end{tabular}
}

\

\

Equations (B) in Section~\ref{sect:intro} are labeled by pairs of positive integers $n,k$, where $n$ is the dimension and $2\le k\le n$. 
{Computer experiments show a remarkable property of  our equations: for given $n$ {\it and} $k$, the highly overdetermined subsystem of linear equations (B) or (M)
 (and assuming implicitly (S), the symmetry property)
has a very large space of solutions, usually much larger than the whole system for given $n$,
which is the conjunction of subsystems for $k=2,\dots,n$ (or just the subsystem for $k=2$, see Lemma~\ref{lemm:chi} in Section~\ref{sect:gen-rel}). 
We have no explanation for this striking fact.}
There are no obvious actions of Hecke operators on the solution spaces $n,k$ individually,  for $k>2$, and it is very surprising 
that the highly overdetermined systems admit any nontrivial solution at all.  


\
\begin{itemize}
\item
$\bQ$-ranks of partial systems $\mathcal B_{n,k}$ and $\mathcal M_{n,k}$ for $k\ge 3$, and for some primes and composite numbers $N$:
\end{itemize}

\

\small{
\begin{tabular}{c|c|c|c|c|c|c|c|c|c||c|c|c|c|c|c|c|c|c}
$N$    &  2  & 3 &  5 &  7 &  11 & 13  & 17 & 19 & 23 & 9 & 12  &  27 & 36  \\
\hline
\hline
\rule{0pt}{3ex}
$\mathcal{B}_{3,3}$ & \bf{1}&2&4&6&12&15&22&27&35&11 & 36  &87&468\\
\rule{0pt}{2ex}
$\mathcal{M}_{3,3}$ &0&\bf{1}&3&3&7&10 &15&18&24\ &9&40&78&480\\
\hline
\rule{0pt}{3ex}
$\mathcal{B}_{4,3}$ & 0 &0&0 & 0 & 0 &0 & 0&0&0&0&\bf{1} &5 &63 \\
\rule{0pt}{2ex}
$\mathcal{M}_{4,3}$ &0&0&0&0&1&2&5&7&12&\bf{1}&5&24&121\\
\hline
\rule{0pt}{3ex}
$\mathcal{B}_{4,4}$ &0&\bf{3}&6&9&17&20&29 & 35&45& 42 &101  & 620 & 2515 \\
\rule{0pt}{2ex}
$\mathcal{M}_{4,4}$&0&\bf{3}& 2  &3 &7&8&13&17&23 &45&123&649&2716\\
\hline
\rule{0pt}{3ex}
$\mathcal{B}_{5,3}$ & 0&0&0&0&0&0&0&0&0&0&0&0& \bf{1} \\
\rule{0pt}{2ex}
$\mathcal{M}_{5,3}$ &0&0&0&0&0&0&0&0&0&0&0&\bf{1}&7\\
\hline
\rule{0pt}{3ex}
$\mathcal{B}_{5,4}$ 0&0&0&0&0&0& 0&0&0&0&\bf{3}&4&55&267\\
\rule{0pt}{2ex}
$\mathcal{M}_{5,4}$&0&0&0&0&1&2&5&7&12 & \bf{5}& 12&122&?\\
\hline
\rule{0pt}{3ex}
$\mathcal{B}_{5,5}$ &\bf{1}&3&9&12&22&26&37&44&56 &30&161&572&?\\
\rule{0pt}{2ex}
$\mathcal{M}_{5,5}$  &0 &\bf{1}&3 &3 &7 &8&13&17&23&17&212 &?&?
\end{tabular}
}

\bibliographystyle{alpha}
\bibliography{cyclic}
\end{document}